\documentclass[12pt]{article}
\usepackage{fancyhdr}

\usepackage{amsmath}
\usepackage{mathtools}
\usepackage{amsfonts}
\usepackage{amsthm}
\usepackage{amssymb}
\usepackage{color}
\usepackage{dsfont}
\usepackage{geometry} 
\newgeometry{tmargin=2.3cm, bmargin=2.8cm, lmargin=1.9cm, rmargin=1.9cm} 

\usepackage{color}

\title{Well-posedness of parabolic equations\\ in the non-reflexive and anisotropic Musielak-Orlicz spaces\\ in the class of renormalized solutions}

\usepackage{authblk}

\author[1]{Iwona Chlebicka\thanks{email address: iskrzypczak@mimuw.edu.pl}}
\author[1,2]{Piotr Gwiazda\thanks{email address: p.gwiazda@mimuw.edu.pl}}
 \author[2]{Anna Zatorska--Goldstein\thanks{email address: azator@mimuw.edu.pl\\ The research of I.C. is supported by NCN grant no. 2016/23/D/ST1/01072. The research of P.G. has been supported by the NCN grant  no. 2014/13/B/ST1/03094.  The research of A.Z.-G. has been supported by the NCN grant  no. 2012/05/E/ST1/03232 (years 2013 - 2017). This work was partially supported by the Simons - Foundation grant 346300 and the Polish Government MNiSW 2015-2019 matching fund.}}
\affil[1]{\small
Institute of  Mathematics, Polish Academy of Sciences, \newline
ul. \'{S}niadeckich 8, 00-656 Warsaw, Poland
}
\affil[2]{\small
Institute of Applied Mathematics and Mechanics,
University of Warsaw, \newline
ul. Banacha 2, 02-097 Warsaw, Poland
}

\newcommand{\barint}{
         \rule[.036in]{.12in}{.009in}\kern-.16in
          \displaystyle\int  }
\def\R{{\mathbb{R}}}

\def\r{{\mathbb{R}}}
\def\N{{\mathbb{N}}}
\def\rn{{\mathbb{R}^{N}}}
\def\rnt{{\mathbb{R}^{N+1}}}

\def\VTM{{V_T^M(\Omega)}}
\def\VTMi{{V_T^{M,\infty}(\Omega)}}
\newcommand{\sg}{{\mathrm{sgn}_0^+}}
\newcommand{\supp}{{\mathrm{supp}}}
\newcommand{\va} {\vec{a}}
\newcommand{\dep} {\delta}
\newcommand{\et} {d}

\def\rp{{\mathbb{R}_{+}}}

\def\ve{{\varepsilon}}
\def\vr{{\varrho}}
\def\vt{{\vartheta^{\tau,r}}}
\def\bt{{\beta^{\tau,r}}}
\def\s{{\sigma}}

\def\bn{{\bar{\nabla}}}

\def\dm{{\underline{m}}}

\def\Mjd{{ M_j^\delta}}
\def\Mss{{(\Mjd(\xi))^{**}}}
\def\Msdx{{(\Mjd(S_\delta(\xi(t,x)) ))^{**}}}
\def\Msd{{(\Mjd)^{**}}}

\def\Ak{{{\cal A }_{k}}}
\def\OT{{{\Omega_T}}}
\def\iOT{{\int_{\OT}}}

\def\iO{{\int_{\Omega}}}
\def\iQd{{\int_{Q_j^\delta\cap\Omega} }}
\def\Qd{{Q_j^\delta}}
\def\tQd{{\widetilde{Q}_j^\delta}}
\def\iQdn{{\int_{Q_j^\delta\cap\{x:\xi_\delta(x)\neq 0\}} }}

\newtheorem{theo}{\bf Theorem}[section]

\newtheorem{lem}{\bf Lemma}[section]
\newtheorem{rem}{\bf Remark}[section]
\newtheorem{defi}{\bf Definition}[section]

\newtheorem{prop}{\bf Proposition}[section]

\newcommand{\wt}{\widetilde}
 
\newcommand{\vp}{\varphi}

\newcommand{\dv}{\mathrm{div}}

\date{}
\begin{document}
\maketitle \sloppy

\thispagestyle{empty}


\parindent 1em

\begin{abstract}

We prove existence and uniqueness of renormalized solutions to   general nonlinear parabolic equation in~Musielak-Orlicz space avoiding growth restrictions. Namely, we consider \begin{equation*}
\partial_t u-\dv A(x,\nabla u)= f\in L^1(\OT),
\end{equation*}
on a Lipschitz bounded domain in $\rn$. The growth of the weakly monotone vector field $A$ is controlled by a generalized nonhomogeneous and anisotropic $N$-function  $M$. The approach does not require any particular type of growth condition of $M$ or its conjugate $M^*$ (neither $\Delta_2$, nor $\nabla_2$). The condition we impose on $M$ is continuity of log-H\"older-type, which results in good approximation properties of the space. However, the requirement of regularity can be skipped in the case of reflexive spaces. The proof of the main results uses truncation ideas, the Young measures methods and monotonicity arguments. Uniqueness results from the comparison principle.
\end{abstract}

\smallskip


\newpage
\section{Introduction}


Our aim is to find a way of proving the existence and uniqueness of renormalized solutions to a strongly nonlinear parabolic equation with $L^1$-data under minimal restrictions on the growth of the leading part of the operator.  We investigate operators $A$, which are monotone, but not necessarily strictly monotone. The~modular function $M$, which controls the growth of the operator, is not assumed to be isotropic, i.e. $M=M(x,\xi)$ not only $M=M(x,|\xi|)$. In turn, we can expect different behaviour of $M(x,\cdot)$ in various directions. We \textbf{do not} require  $M\in\Delta_2$, nor $M^*\in\Delta_2$, nor~any particular growth of $M$, such as $M(x,\xi)\geq c|\xi|^{1+\nu}$ for $\xi>\xi_0$. In general, if the modular function has a growth of type far from being polynomial (e.g. exponential), it entails analytical difficulties and significantly restricts good properties of the space, such as separability or reflexivity, as well as admissible classical tools. In order to relax the conditions on the growth we require the log-H\"older-type regularity of the modular function (cf.~condition (M)), which can be skipped in  reflexive  spaces.

\smallskip

We study  the problem
\begin{equation}\label{intro:para}\left\{\begin{array}{ll}
\partial_t u-\dv A(x,\nabla u)= f(t,x) & \ \mathrm{ in}\  \OT,\\
u(t,x)=0 &\ \mathrm{  on} \ \partial\Omega,\\
u(0,x)=u_0(x) & \ \mathrm{ in}\  \Omega,
\end{array}\right.
\end{equation}
where $[0,T]$ is a finite interval, $\Omega$ is a bounded Lipschitz domain in $ \rn$, $\OT=(0,T)\times \Omega$, $N>1$, $f\in L^1(\OT)$, $u_0\in L^1(\Omega)$, within two classes of functions:
\[ \begin{split} \VTM  &=\{u\in L^1(0,T;W_0^{1,1}(\Omega)):\ \nabla u\in L_M(\OT;\rn)\},\\
 \VTMi 
 & =\VTM \cap L^\infty(0,T; L^2(\Omega)).\end{split}\]
The space $L_M$ (Definition~\ref{def:MOsp}) is equipped with the modular function $M$ being an $N$-function (Definition~\ref{def:Nf}) controlling the growth of $A$. 
 
We consider $A$~belonging to an Orlicz class with respect to the second variable. Namely, we assume that function $A:\Omega\times\rn\to\rn$  satisfies the following conditions.
\begin{itemize}
\item[(A1)] $A$ is a Carath\'eodory's function.
\item[(A2)] There exists an $N$-function  $M:\Omega\times\rn\to\r$ and a constant $c_A\in(0,1]$ such that for all $\xi\in\rn$ we have
\[A(x,\xi)\xi\geq c_A\left(M(x,\xi)+M^*(x,A(x,\xi))\right),\]
where $M^*$ is conjugate to $M$ (see Definition~\ref{def:conj}).
\item[(A3)] For all $\xi,\eta\in\rn$ and $x\in\Omega$ we have 
\[(A(x,\xi) - A(x, \eta)) \cdot (\xi-\eta)\geq 0.\]
\end{itemize}

Unlike other studies of the Musielak-Orlicz spaces e.g.~\cite{gwiazda-ren-ell,gwiazda-ren-para,hhk,mmos2013} instead of growth conditions we assume  regularity of~$M$.
\begin{itemize} 
\item[(M)]  Let us consider a family of $N$-dimensional cubes   covering the set $\Omega$. Namely, a~family $\{\Qd\}_{j=1}^{N_\delta}$ consists of closed cubes of edge $2\delta$, such that  $\mathrm{int}\Qd\cap\mathrm{int} Q^\delta_i=\emptyset$ for $i \neq j$ and $\Omega\subset\bigcup_{j=1}^{N_\delta}\Qd$. Moreover, for each
cube $\Qd$ we define the cube $\tQd$ centered at the same point and with parallel corresponding edges of length $4\delta$.   Assume that there exist constants $a,c,\delta_0 >0$, such that for all $\delta<\delta_0$, $x\in\Qd$ and all $\xi\in\rn$ we have
\begin{equation}
\label{M2}  
\frac{M(x,\xi )}{\Mss} 
\leq c \left(1+ |\xi|^{-\frac{a}{\log(3\delta\sqrt{N})}} \right),
\end{equation}
where \begin{equation}
\label{Mjd}\Mjd(\xi ):= \inf_{x\in \widetilde{Q}_j^\delta\cap\Omega}M(x,\xi ),
\end{equation} 
while $\Mss=({(\Mjd(\xi))}^*)^* $ is the greatest convex minorant of $\Mjd(\xi )$ (coinciding with the~second conjugate cf. Definition~\ref{def:conj}).

Moreover, suppose that for every measurable set $G\subset\Omega$ and every $z\in\rn$ we have
\begin{equation}
\label{ass:M:int}\int_G M(x,z)dx<\infty.
\end{equation}
\end{itemize}
 In further parts of the introduction we describe the cases when the above condition is not necessary. Let us only point out that to get (M) in the isotropic case, i.e. when we consider $M(x,\xi)=M(x,|\xi|)$, it suffices to assume log-H\"older-type condition with respect to $x$, namely \eqref{M2'},  cf.~Lemma~\ref{lem:Mass}.

We apply the truncation techniques. Let truncation $T_k(f)(x)$ be defined as follows\begin{equation}T_k(f)(x)=\left\{\begin{array}{ll}f & |f|\leq k,\\
k\frac{f}{|f|}& |f|\geq k.
\end{array}\right. \label{Tk}
\end{equation}

We call a function $u$ a renormalized solution to~\eqref{intro:para}, when it satisfies the following conditions.\begin{itemize}
\item[(R1)] $u\in L^1(\OT)$ and for each $k>0$  \[T_k(u)\in \VTM  ,\qquad A(\cdot,\nabla T_k(u))\in L_{M^*}(\Omega;\rn).\]
\item[(R2)] For every $h\in C^1_0(\R)$ and all $\varphi\in \VTMi$, such that $\partial_t\vp\in L^\infty(\Omega_T)$ and $\vp(\cdot,x)$ has a compact support in $[0,T)$ for a.e. $x\in\Omega$, we have
\begin{equation*}
-\int_{\OT} \left(\int_{u_0(x)}^{u(t,x)}h(\sigma)d\sigma\right)  \partial_t \vp\ dx\,dt+\int_{\OT}A(x,\nabla u)\cdot\nabla(h(u)\vp) \,dx\,dt=\int_{\OT}f h(u)\vp \,dx\,dt.
\end{equation*}
\item[(R3)] $\displaystyle{\int_{ \{l<|u|<l+1\}}A(x,\nabla u)\cdot\nabla u\, dx\,dt\to 0}$ as $l\to\infty$.
\end{itemize} 

Our main result reads as follows.
\begin{theo}\label{theo:main} Suppose $[0,T]$ is a finite interval, $\Omega$ is a bounded Lipschitz domain in $ \rn$,   $N>1$, \\ $f\in L^1(\OT)$, $u_0\in L^1(\Omega)$. Let an $N$-function $M$ satisfy assumption (M) and function $A$ satisfy assumptions (A1)-(A3). Then there exists the unique renormalized weak solution to the problem~\eqref{intro:para}. 
 Namely, there exists $u \in \VTM$, which satisfies (R1)-(R3). 
\end{theo}

\begin{rem}[cf.~\cite{martin}] \rm When the modular function has a special form we can simplify our assumptions. In the case of $M(x,\xi)=M(x,|\xi|)$, via Lemma~\ref{lem:Mass}, we replace condition (M) in the above theorem by log-H\"older continuity of M, cf.~\eqref{M2'}. If $M$ has a form 
\[M(x,\xi)=\sum_{i=1}^jk_i(x)M_i(\xi)+M_0(x,|\xi|),\quad j\in\N,\]
instead of~whole (M) we assume only that $M_0$ is log-H\"older continuous~\eqref{M2'}, all $M_i$ for $i=1,\dots,j$ are $N$-functions and all $k_i$ are nonnegative and satisfy $\frac{k_i(x)}{k_i(y)}\leq C_i^{\log \frac{1}{|x-y|}}$ with $C_i>0$ for $i=1,\dots,j$.
\end{rem} 
 
 Our framework admitts the following examples.\begin{itemize}
\item  $M(x,\xi)=|\xi|^{p(x)}$ with log-H\"older $p:\Omega\to[p^-,p^+]$, where $p^-=\inf_{x\in\Omega}p(x)>1$ and $p^+=\sup_{x\in\Omega}p(x)<\infty,$ then $V_0^M=  W_0^{1,p(\cdot)}(\Omega)$ and we admitt $A(x,\xi)=|\xi|^{p(x)-2}\xi$ ($p(\cdot)$-Laplacian case) as well as $A(x,\xi)=\alpha(x)|\xi|^{p(x)-2}\xi$   with $0<<\alpha(x)\in L^\infty(\Omega)\cap C(\Omega);$
\item $M(x,\xi)=\sum_{i=1}^N|\xi_i|^{p_i(x)}$, where $\xi=(\xi_1,\dots,\xi_N)\in\rn$, log-H\"older functions $p_i:\Omega\to[p_i^-,p_i^+]$, $i=1,\dots,N$, where $p_i^-=\inf_{x\in\Omega}p_i(x)>1$ and $p_i^+=\sup_{x\in\Omega}p_i(x)<\infty,$ then $V_0^M= W_0^{1,\vec{p}(\cdot)}(\Omega)$ and we admitt  \[A(x,\xi)=\sum_{i=1}^N\alpha_i(x)|\xi_i|^{p_i(x)-2}\xi\quad\text{ with }\quad 0<<\alpha_i(x)\in L^\infty(\Omega)\cap C(\Omega).\]
\end{itemize} 
Our assumption that $M,M^*$ are $N$-functions (Definition~\ref{def:Nf})  in the variable exponent setting   restrict us to the case of $1<p_-\leq p(x)\leq p^+<\infty$.

\subsubsection*{State of art}

The problems like~\eqref{intro:para} are very well understood, when $A$ is independent of the spacial variable and has a polynomial growth. In~particular, there is vast literature for analysis of the case involving the $p$-Laplace operator $A(x,\xi)=|\xi|^{p-2}\xi$ and  problems  stated in the Lebesgue space setting (the modular  function is then $M(x,\xi)=|\xi|^p$). There is a wide range of directions in~which the polynomial growth case has been developed including the variable exponent, Orlicz, weighted and double-phase spaces. The Musielak-Orlicz spaces, which include in particular all of the mentioned types of spaces,  have been studied systematically starting from~\cite{Musielak,Sk1,Sk2}. Investigations of nonlinear boundary value problems in~non-reflexive Orlicz-Sobolev-type setting was initiated by Donaldson~\cite{Donaldson} and continued by Gossez~\cite{Gossez2,Gossez3,Gossez}. For a~summary of~the~results we refer to~\cite{Mustonen} by Mustonen and Tienari. The generalization to~the~case of~vector Orlicz spaces with the anisotropic modular  function,  but independent of spacial variables was investigated in~\cite{Gparabolic}. Let us note that  the variable exponent Lebesgue spaces (when $M(x,  \xi ) = | \xi |^{p(x)}$  with $1 < p_{\rm min} \leq p(x) \leq  p^{\rm max} < \infty $), as well as~weighted or the double phase space (when $M(x,  \xi ) = | \xi |^{p}+a(x)|\xi|^q$  with $a\geq 0$ and $1 < p,q < \infty $) are still reflexive, which help significantly in approximation properties of the space via Mazur's Lemma. Unlike them the generalised Musielak-Orlicz spaces in~general fail in~being reflexive. We aim in~providing the existence result also in this non-relfexive cases. Our approach essentially involves the theory arising from fluids mechanics~\cite{gwiazda-non-newt,gwiazda-tmna,gwiazda2,Aneta}. Let us indicate that these papers provide many facts useful in~analysis of Musielak--Orlicz spaces. For other recent developments of~the framework of the spaces let us refer e.g. to~\cite{hhk,hht,mmos:ap,mmos2013}.

Partial differential equations with the right-hand side in $L^1$ received special attention. DiPerna and Lions investigating the~Boltzmann equation in order to deal with this challenge introduced the notion of renormalized solutions in~\cite{diperna-lions}. Other seminal ideas for problems with $L^1$-data comes from~\cite{bgSOLA}, where the solution is obtained as a limit of approximation, and~\cite{dall}, where entropy solutions are studied. Let us stress that the mentioned notions coincides. In~\cite{DrP} the authors show equivalence between entropy and renormalised solutions for problems with polynomial growth. Meanwhile, the corresponding result in the variable exponent and the Orlicz settings are provided together with the proofs of~the existence of renormalized solutions in~\cite{ZZ,zhang17}, respectively.

 In the elliptic setting the foundations of the studies on renormalized solutions, providing results for operators with polynomial growth, were laid by Boccardo et.~al.~\cite{boc-g-d-m},  Dall'Aglio~\cite{dall} and Murat~\cite{murat}. In the parabolic setting, renormalized solutions were studied first in~\cite{B,BMR99,BR,BR98,boc-ors1} and further~\cite{BMR,DP,DrP,P,PPP}. These studies are continued under weaker assumptions on the data~\cite{BPR13,BCM,DMOP}. Lately, generalising the setting, renormalized solutions to parabolic problems have been considered in the variable exponent setting~\cite{BWZ,LG,ZZ} and in the model of thermoviscoelasticity~\cite{ch-o}. For very recent results on entropy and renormalised solutions, we refer also to~\cite{ch-o,F15,MMR17,zhang17}. This issue in~parabolic problems in non-reflexive Orlicz-Sobolev spaces are studied in~\cite{HMR,MMR17,R10,zhang17}, while in the nonhomogeneous and non-reflexive Musielak-Orlicz spaces in~\cite{gwiazda-ren-para} (under certain growth conditions on the modular function).
 
\subsubsection*{Approximation in Musielak-Orlicz spaces} 
The highly challenging part of analysis in the general Musielak-Orlicz spaces is giving a relevant structural condition implying approximation properties of the space. However, we are equipped not only with the weak-* and strong topology of the gradients, but also with  the modular topology.

 In the mentioned existence results even in the case, when the growth conditions imposed on~the~modular  function were given by a~general $N$-function, besides the growth condition on~$M^*$, also $\Delta_2$-condition on $M$ was assummed (which entails separability of~$L_{M^*}$, see~\cite{Aneta}). It results further in density of smooth functions in $L_M$ with respect to the weak-$*$ topology. In the case of~classical Orlicz spaces, the crucial density result was provided by Gossez~\cite{Gossez}, improved for the vector Orlicz spaces in~\cite{Gparabolic}. However, the case of $x$--dependent log-H\"{o}lder continuous modular functions was claimed to cover first in~\cite{BenkiraneDouieb}, the proof involved an essential gap. In~\cite[(31)]{BenkiraneDouieb} the Jensen inequality is used for the infimum of convex functions, which obviously is not necessarily convex. We fix the proof in the elliptic case in~\cite[Theorem~2.2]{pgisazg1} and in the parabolic case in Theorem~\ref{theo:approx} below, changing slightly assumptions.
 
 The Musielak-Orlicz spaces equipped with the modular function satisfying $\Delta_2$-condition (cf.~Definition~\ref{def:D2}) have strong properties, nonetheless there is a vast range of $N$-functions not satisfying it, which we want to cover, e.g.
\begin{itemize}
\item  $M(x,\xi)=a(x)\left( \exp(|\xi|)-1+|\xi|\right)$;
\item  $M(x,\xi)= |\xi_1|^{p_1(x)}\left(1+|\log|\xi||\right)+\exp(|\xi_2|^{p_2(x)})-1$, when $(\xi_1,\xi_2)\in\R^2$ and $p_i:\Omega\to[1,\infty]$. It is a model example to imagine what we mean by anisotropic modular function. 
\end{itemize}
 
Let us discuss our assumption (M). First we shall stress that it is applied only in~the~proof of~approximation result (Theorem~2.2). When we deal with the space equipped with the approximation properties, we can simply skip (M). Namely, this is the case e.g. of the following modular functions:
\begin{itemize}
\item $M(x,|\xi|)=|\xi|^p+a(x)|\xi|^q$, where $1\leq p<q<\infty$ and function $a$ is nonnegative a.e. in $\Omega$ and $a\in L^\infty(\Omega)$, covering the celebrated case of the double-phase spaces~\cite{min-double-reg};
\item $M(x,|\xi|)=|\xi|^{p(x)}+a(x)|\xi|^{q(x)}$, where $1<< p<q<<\infty$ and function $a$ is nonnegative a.e. in $\Omega$ and $a\in L^\infty(\Omega)$, covering the weighted and double-phase variable exponent case;
\item $M(x,\xi)=M_1(\xi)+a(x)M_2(\xi)$, where $M_1,M_2$ satisfy conditions  $\Delta_2$ and $\nabla_2$, moreover a~function $a$ is nonnegative a.e. in $\Omega$ and $a\in L^\infty(\Omega)$.
\end{itemize}
In the  above cases (and in the case of any other reflexive space) whenever in the proof we apply an approximation by a sequence of smooth functions converging modularly, provided in~Theorem~\ref{theo:approx}, we can use instead a strongly converging  affine combination of the weakly converging sequence (ensured in   reflexive Banach spaces via Mazur's Lemma). Indeed, even if we do not deal with modular density of smooth functions, due to our approximation scheme the solution is in strong closure of smooth functions.

In the variable exponent case typical assumption resulting in approximation properties of the space is log-H\"older continuity of the exponent. In the isotropic case (when $M(x,\xi)=M(x,|\xi|)$) 
Lemma~\ref{lem:Mass} shows that to get (M), it suffices to   impose on $M$ continuity condition of log-H\"older-type with respect to $x$, namely for each $\xi\in\rn$ and $x,y,$ such that $|x-y|<\frac{1}{2}$ we have\begin{equation}
\label{M2'} \frac{M(x,\xi)}{M(y,\xi)}\leq\max\left\{ |\xi|^{-\frac{a_1}{\log|x-y|}}, b_1^{-\frac{a_1}{\log|x-y|}}\right\},\ \text{with some}\ a_1>0,\,b_1\geq 1.
\end{equation}  Note that condition~\eqref{M2'} 
 for $M(x,\xi)=|\xi|^{p(x)}$ relates to the log-H\"older continuity condition for the variable exponent $p$, namely there exists $a>0$, such that for $x,y$ close enough and  $|\xi|\geq 1$
\[|p(x)-p(y)|\leq \frac{a}{\log\left(\frac{1}{|x-y|}\right)}.\]
Indeed, whenever $|\xi|\geq 1$
\[ \frac{M(x,\xi)}{M(y,\xi)}= \frac{|\xi|^{p(x)}}{|\xi|^{p(y)}}=|\xi|^{p(x)-p(y)}\leq |\xi|^\frac{a}{\log\left(\frac{1}{|x-y|}\right)}=|\xi|^{-\frac{a}{\log {|x-y|} }}.\]

  There are several types of understanding generalisation of~log-H\"older continuity to the case of~general $x$-dependent isotropic modular functions (when $M(x,\xi)=M(x,|\xi|)$). The important issue is the interplay between types of continuity with respect to each of the variables separately. Besides our condition~\eqref{M2'} (sufficient for (M) via Lemma~\ref{lem:Mass}), we refer to the approaches of~\cite{hhk,hht} and~\cite{mmos:ap,mmos2013}, where the authors deal with the modular function of the form $M(x,\xi)=|\xi|\phi(x,|\xi|)$. We proceed without their doubling assumptions ($\Delta_2$). Since we are restricted to bounded domains, condition $\phi(x,1)\sim 1$ follows from our definition of $N$-function (Definition~\ref{def:Nf} ). As for the types of continuity,  in~\cite{mmos:ap,mmos2013} the authors restrict themselves to the case when $\phi(x,|\xi|)\le c \phi(y,|\xi|)$ when $|\xi|\in [1,|x-y|^{-n}].$ This condition implies~\eqref{M2'} and consequently~(M). Meanwhile in~\cite{hhk,hht}, the proposed condition yields  $\phi(x, b|\xi|)\le \phi(y,|\xi|)$ when $\phi(y,|\xi|)\in [1, |x-y|^{-n}],$ which does not imply~\eqref{M2'} directly. However, we shall mention that all three conditions are of the same spirit and balance types of continuity with respect to each of the variables separately.

\subsubsection*{Our approach}
Studying the problem~\eqref{intro:para} we face the challenges resulting from the lack of the growth conditions and handling with general $x$-dependent and anisotropic $N$-functions. The space we deal with is, in~general, neither separable, nor reflexive.  Lack of~precise control on the growth of~the  leading part of the operator, together with the low integrability  of~the~right-hand side triggers noticeable difficulties in studies on convergence of approximation. Note that the growth of the leading part of the operator is naturally driven by growth-coercivity condition~(A2). It is explained in detail in preliminaries of~\cite{pgisazg1} concerning the elliptic case.

Our methods  employ the framework developed in~\cite{pgisazg1,gwiazda-ren-ell,gwiazda-ren-cor,gwiazda-ren-para} studying elliptic and parabolic problems.   
Resigning from imposing $\Delta_2$-condition on the conjugate of the modular function  complicates understanding of the dual pairing. As a~further consequence of relaxing growth condition on the modular function, we cannot use classical results, such as the Sobolev embeddings,~the~Rellich-Kondrachov compact embeddings, or the~Aubin-Lions Lemma. Unlike the other studies, we put regularity restrictions on the modular function instead of the growth conditions, which however can be skipped in the case of reflexive spaces. From this point of view our paper is a natural continuation of elliptic approach of~\cite{pgisazg1}. On the other hand, current research involves   essential new ideas. Due to the appearance of the evolution term, a challenging part is the proof of the integration-by-parts formula (Lemma~\ref{lem:intbyparts}). The identification of the limit of $A(x,\nabla T_k(u_n))$ is much more complicated than in the elliptic setting. Moreover, to by-pass the~Aubin-Lions Lemma from the corresponding study~\cite{gwiazda-ren-para} involving growth conditions used to get almost everywhere convergence of the solutions to the truncated problem, we provide the comparison principle. We apply it twice: to get the mentioned almost everywhere convergence  and to obtain uniqueness.

Above we stress how demanding is approximation in the general Musielak-Orlicz spaces. In~Theorem~\ref{theo:approx} we provide a parabolic version of approximation result of~\cite{pgisazg1} using ideas of~\cite{ASGcoll}. Theorem~\ref{theo:approx} is a key tool in the proof of weak renormalised formulation given in Lemma~\ref{lem:intbyparts}. We provide the comparison principle in Proposition~\ref{prop:comp-princ}. The main goal, i.e. the existence of renormalized solutions to  general nonlinear parabolic equation, is given in Theorem~\ref{theo:main} above and proven in Section~\ref{sec:mainproof}. 

Let us summarize the scheme of the main proof contained in Sections~\ref{sec:constr} and~\ref{sec:mainproof}. The first step is to show existence of weak solutions to the bounded regularized problem (Proposition~\ref{prop:reg-bound}), while in the second step we prove existence of weak solutions $u_n$ to the non-regularized problem with bounded data (Proposition~\ref{prop:bound}) using  the Minty-Browder monotonicity trick.  The third step is establishing certain types of convergence of truncations of a solution $T_k(u_n)$ (Proposition~\ref{prop:convTk}). The fourth step is devoted to the radiation control condition relating to (R3), but for $u_n$ (Proposition~\ref{prop:contr:rad:n}). In~the~fifth one the comparison principle is applied to obtain almost everywhere convergence of $u_n$. In step~6 we localize $A(x,\nabla T_k(u))$ as the weak-* limit in $L_{M^*}$ of~$A(x,\nabla T_k(u_n))$ (Proposition~\ref{prop:convTkII}), involving the methods of~\cite{ammar,gwiazda-ren-ell,gwiazda-ren-para,Landes-meth} and the monotonicity trick. Section~\ref{sec:mainproof} finally concludes the proof of~existence of renormalized solutions. We motivate weak $L^1$-convergence of $A(x,\nabla T_k(u_n))\cdot \nabla T_k(u_n)$ via the Young measures' method. 

In the end we include appendices providing basic definitions, auxiliary results, fundamental theorems, proofs of approximation result, and weak renomalised formulation.

\section{Preliminaries}
  
In this section we give only the general preliminaries concerning the setting.  We assume $\Omega\subset\rn$ is a bounded Lipschitz domain, $\Omega_T=(0,T)\times\Omega,$ $\Omega_\tau=(0,\tau)\times\Omega.$ If $V\subset\R^K$, $K\in\N$, is a bounded set, then $C_c^\infty(V)$ denote the class of smooth functions with support compact in $V$.  We denote positive part of function signum by $\sg(s)=\max\{0,s/|s|\}$.

\subsubsection*{Classes of functions}
 
\begin{defi}\label{def:MOsp} Let $M$ be an $N$-function (cf.~Definition~\ref{def:Nf} in Appendix~\ref{ssec:Basics}).\\ We deal with the three  Orlicz-Musielak classes of functions.\begin{itemize}
\item[i)]${\cal L}_M(\Omega;\rn)$  - the generalised Orlicz-Musielak class is the set of all measurable functions\\ $\xi:\Omega\to\rn$ such that
\[\int_\Omega M(x,\xi(x))\,dx<\infty.\]
\item[ii)]${L}_M(\Omega;\rn)$  - the generalised Orlicz-Musielak space is the smallest linear space containing ${\cal L}_M(\Omega;\rn)$, equipped with the Luxemburg norm 
\[||\xi||_{L_M}=\inf\left\{\lambda>0:\int_\Omega M\left(x,\frac{\xi(x)}{\lambda}\right)\,dx\leq 1\right\}.\]
\item[iii)] ${E}_M(\Omega;\rn)$  - the closure in $L_M$-norm of the set of bounded functions.
\end{itemize}
\end{defi}
Then 
\[{E}_M(\Omega;\rn)\subset {\cal L}_M(\Omega;\rn)\subset { L}_M(\Omega;\rn),\]
the space ${E}_M(\Omega;\rn)$ is separable and $({E}_M(\Omega;\rn))^*=L_{M^*}(\Omega;\rn)$, see~\cite{gwiazda-non-newt,Aneta}.

Under the so-called $\Delta_2$-condition (Definition~\ref{def:D2}) we would be equipped with stronger tools. Indeed,  if $M\in\Delta_2$, then
\[{E}_M(\Omega;\rn)= {\cal L}_M(\Omega;\rn)= {L}_M(\Omega;\rn)\]
and $L_M(\Omega;\rn)$ is separable. When both  $M,M^*\in\Delta_2$, then $L_M(\Omega;\rn)$ is separable and reflexive, see~\cite{GMWK,gwiazda-non-newt}. We face the problem \textbf{without} this structure.

\begin{rem} Definition~\ref{def:Nf} (see points 3 and 4)  implies 
$\lim_{|\xi|\to \infty}\inf_{x\in\Omega}\frac{M^*(x,\xi)}{|\xi|}=\infty$ and 
  $\inf_{x\in\Omega}M^*(x,\xi)>0$ for any $\xi\neq 0$. Then, consequently, Lemma~\ref{lem:M*<M} ensures 
\begin{equation}
\label{LinfinEM}L^\infty(\Omega;\rn) \subset E_M(\Omega;\rn)\qquad\text{and}\qquad L^\infty(\OT;\rn) \subset E_M(\OT;\rn).\end{equation} 
\end{rem} 

\subsubsection*{Auxiliary functions}
We list here special forms of auxiliary functions used in the proofs.

 Let $\psi_l:\R\to\R$ be given by
\begin{equation}
\label{psil}\psi_l(s):=\min\{(l+1-|s|)^+,1\}.
\end{equation} 

Let two-parameter family of functions $\vt:\R\to\R$ be defined by  
\begin{equation}
\label{vt}\vt(t):=\left(\omega_r * \mathds{1}_{[0,\tau)}\right)(t),
\end{equation}where $\omega_r$ is a standard regularizing kernel, that is $\omega_r\in C_c^\infty(\R)$, $\supp\,\omega_r\subset (-r,r)$. Note that $\supp\vt=[-r,\tau+r).$ In particular, for every $r$ there exists $r_\tau$, such that for all $r<r_\tau$ we have $\vt\in C_c^\infty([0,T))$. 

We consider a one-parameter family of nonincreasing functions $\phi_{r}\in C_c^\infty([0,T))$ given by
\begin{equation}
\label{phidelta}\phi_{r}(t):=\left\{\begin{array}{ll}
0& \text{for }t\in[T-{r},T],\\
1& \text{for }t\in[0,T-2{r}].
\end{array}\right.
\end{equation}

 Let us consider $g:\R\times\Omega\to\R^K,$ $K\in\N$. When $\vr_\mu(s)=\mu e^{-\mu s}\mathds{1}_{[0,\infty)}(s)$, $\mu>0$, the regularized function $g_\mu:\R\times\Omega\to\R$ is defined by \[g_\mu(t,x):=(\vr_\mu*{g})(t,x),\]
where $*$ stands for the convolution is in the time variable. Then
\begin{equation}
\label{gmu} g_\mu(t,x) =\mu\int_{-\infty}^t e^{\mu(s-t)}g(s,x)\,ds.
\end{equation}
Note that the linear mapping $g\mapsto g_\mu$ is bounded from $L^\infty(\OT)$ into $L^\infty(\OT)$, as well as  from $L_M(\OT;\rn)$  into $L_M(\OT;\rn)$. Moreover, $g_\mu$ is a unique solution to 
\begin{equation}
\label{gmu:ode}
\begin{split}\partial_t g_\mu+ \mu(g_\mu-g) =0&\quad \text{ in }{\cal D'}(\R\times\Omega),\\
g_\mu(0,x) =g_0^{\mu}&\quad\text{ a.e. in }\Omega.\end{split}
\end{equation}

\subsubsection*{Approximation}

The following result coming from combined ideas of~\cite{pgisazg1} and~\cite{ASGcoll} is proven in Appendix~\ref{ssec:Appr}.
\begin{theo}[Approximation theorem]\label{theo:approx} 
Let $\Omega$ be a Lipschitz domain and an $N$-function $M$~satisfy condition (M). Then for any $\vp$ such that $\vp\in \VTMi$  there exists a sequence $\{\vp_\ve\}_{\ve>0}\in L^\infty (0,T; C_c^\infty(\Omega))$ converging modularly to $\vp$ when $\ve\to 0$, i.e. such that $\nabla\vp_\ve\xrightarrow[]{M}\nabla \vp$ in $\Omega_T$ and $ \vp_\ve\to  \vp$ strongly in $L^1(\Omega_T)$.
\end{theo} 

\subsubsection*{Weak renormalised formulation}
  
 Adapting the framework of~\cite{gwiazda-ren-para} to our setting, we have the following formulation, whose proof is given in Appendix~\ref{ssec:Weak}.
\begin{lem}\label{lem:intbyparts}
Suppose $u:\Omega_T\to\r$ is a measurable function such that for every $k\geq 0$, $T_k(u)\in \VTM$, $u(t,x)\in L^\infty([0,T];L^1(\Omega))$. Let us assume that there exists $u_0\in L^1(\Omega)$ such that $u_0(x):=u(0,x)$. Furthermore, assume that there exist $A\in L_{M^*}(\Omega_T;\rn)$  and $F\in L^1(\Omega_T)$ satisfying
\begin{equation}
\label{eq:lem:int-by-parts-1}
-\int_{\OT}(u-u_0)\partial_t \vp \,dx\,dt+\int_{\OT}A\cdot \nabla\vp \,dx\,dt=\int_{\OT}F\, \vp \,dx\,dt,\qquad \forall_{\vp\in{C_c^\infty}([0,T)\times \Omega)}.
\end{equation}
Then 
\begin{equation*}
-\int_{\OT} \left(\int_{u_0}^u h(\s)d\s\right) \partial_t \xi \ \,dx\,dt+\int_{\OT}A\cdot \nabla (h(u)\xi) \,dx\,dt=\int_{\OT}F h(u)\xi \,dx\,dt
\end{equation*}
holds for all $h\in W^{1,\infty}(\r)$, such that $\supp (h')$ is compact and all $\xi\in \VTMi$, such that $\partial_t\xi\in L^\infty(\OT)$ and $\supp\xi(\cdot,x)\subset[0,T)$ for a.e. $x\in\Omega$, in particular for  $\xi \in C_c^\infty([0,T)\times\overline{\Omega})$. 
\end{lem}
 
\subsubsection*{Comparison principle}

The comparison principle we provide below is the consequence of choice proper family of test functions. The result will be used in the proof of almost everywhere convergence of $(u_n)_n$ and uniqueness of solutions. The proof is presented in Appendix~\ref{ssec:Comp}.

\begin{prop}\label{prop:comp-princ}
Suppose $v^1,v^2$ are renormalized solutions  to\[\left\{\begin{array}{l}
v^1_t-{\dv }A(x,\nabla v^1)= f^1\in L^1(\OT),\\
v^1(0,x)=v^1_0(x)\in L^1(\Omega)
\end{array}\right.\qquad \left\{\begin{array}{l}
v^2_t-{\dv }A(x,\nabla v^2)= f^2\in L^1(\OT),\\
v^2(0,x)=v^2_0(x)\in L^1(\Omega),
\end{array}\right.\]where $f^1\leq f^2$ a.e. in~$\OT$ and $ v^1_0 \leq v^2_0$ in~$\Omega.$ Then $v^1\leq v^2$ a.e. in~$\OT$.
\end{prop}

\section{The construction and convergence}\label{sec:constr}
  The proof is divided into several steps. We start with the proof of existence to a regularized problem and then to the problem with bounded data, while afterwards we verify the convergence.

\subsubsection*{Step 1. Regularized truncated problem}

 We apply the general method of~\cite{ElMes} leading to existence. Let   $m:\rn\to\R$ be a radially summetric function, i.e. $m(\xi)=\overline{m}(|\xi|)$ with some  $\overline{m}:\R\to\R$. We say that $m$ \textbf{grows essentially more rapidly} than $M$ if \[ {\overline{M}(s)}/{\overline{m}(s)}\to 0\quad\text{ as }\quad s\to\infty\quad\text{ for }\quad {
\overline{M} (s) = \sup_{\{x\in\Omega,\ |\xi|=s\}}
M (x, \xi)}.\]

Let us point out that when $\Omega$ has finite measure and $m$ grows essentially more rapidly than $M$, we have
\begin{equation}
\label{LminEM}
L_m(\Omega;\rn)\subset E_M(\Omega;\rn)\quad\text{and}\quad 
L_m(\Omega_T;\rn)\subset E_M(\Omega_T;\rn).
\end{equation} 

Recall that we use notation $\nabla$ for a gradient with respect to the spacial variable. Let us introduce also notation $\bn:=\nabla_\xi$. Using it  $\bn m(\xi) =\nabla_\xi \overline{m}(|\xi|)=\xi\overline{m}'(|\xi|)/|\xi|$. Observe that it gives equality in the Fenchel--Young inequality in the following way
\begin{equation}\label{FYeq}
\bn m(\xi)\cdot\xi = \overline{m}( |\xi|)+\overline{m}^*(|\bn m(\xi)|).
\end{equation}
Taking an arbitrary $N$-function $m$ which grows essentially more rapidly than  $M$ we observe that  $m$ is strictly monotone as a gradient of a strictly convex function, i.e.
\begin{equation}\label{msmon}
(\bn m(\xi)-\bn m(\eta))(\xi-\eta)>0\qquad \forall_{\xi,\eta\in\rn}.
\end{equation}

The following proposition yields the existence of solutions to a regularized problem. 
\begin{prop}\label{prop:reg-bound}  Assume  an $N$-function $M$ satisfies assumption (M), a vector field $A$ satisfies assumptions (A1)-(A3), and $m$ is a~radially symmetric function growing essentially more rapidly than $M$ satisfying~\eqref{msmon}.

We consider a regularized operator given by
\begin{equation}\label{Atheta}
A_\theta(x,\xi):=A(x,\xi)+\theta \bn m(\xi)\qquad \forall_{x\in\rn,\,\xi\in\rn}.
\end{equation} 

 Let $f\in L^1(\Omega_T)$ and $u_0\in L^1(\Omega)$. Then for every $\theta\in(0,1]$ and $n\in\N$ there exists a weak solution to the problem  \begin{equation}\label{eq:reg-bound}
\left\{\begin{array}{ll}
\partial_t u_n^\theta-\dv A_\theta(x,\nabla u_n^\theta)= T_n (f) & \ \mathrm{ in}\  \OT,\\
u_n^\theta(t,x)=0 &\ \mathrm{  on} \ (0,T)\times\partial\Omega,\\
u_n^\theta(0,\cdot)=u_{0,n}(\cdot)=T_n(u_0)\in L^1(\Omega) & \ \mathrm{ in}\  \Omega.
\end{array}\right.
\end{equation} Namely, there exists $u_n^\theta \in C([0,T];L^2(\Omega))\cap L^1(0,T; W^{1,1}_{0}(\Omega))$ with $\nabla u_n^\theta\in L_m(\Omega_T;\rn)$, such that 
\begin{equation}\begin{split}\label{weak-reg-bound}
&-\iOT u_n^\theta\partial_t \vp dx\, dt+\iO u_n^\theta(T)\vp(T)\, dx- \iO u_n^\theta(0)\vp(0)\, dx+\iOT A_\theta(x,\nabla u_n^\theta)\cdot \nabla \vp\, dx\, dt \\&= \iOT T_n(f)\vp\, dx\, dt 
\end{split}\end{equation} 
holds for $\vp\in C^\infty([0,T];C_c^\infty(\Omega))$.

Moreover, the energy equality is satisfied, i.e.
\begin{equation}
\label{en:eq}
\frac{1}{2}\iO (u_n^\theta(\tau))^2\, dx-
\frac{1}{2}\iO (u_{0,n} )^2\, dx+\int_{\Omega_\tau} A_\theta(x,\nabla u_n^\theta)\cdot \nabla u_n^\theta\, dx\, dt=  \int_{\Omega_\tau} T_n(f) u_n^\theta\, dx\, dt,
\end{equation}
where $\tau\in[0,T]$.
\end{prop}
\begin{proof} To apply \cite[Theorem~2]{ElMes}, we verify the assumption therein.  Monotonicity condition \cite[(9)]{ElMes} follows from~\eqref{msmon} and~(A3), while \cite[(10)]{ElMes} from~\eqref{FYeq} and~\eqref{Atheta}.

It suffices now to show that we have upper bound~\cite[(7)]{ElMes}   on the growth of $A_\theta$, i.e.
\begin{equation}\label{A-ElMes7}
|A_\theta(x,\nabla u^\theta_n)|\leq k_1(x)+c_1(\overline{m}^*)^{-1}(\overline{m}(|c_2 \nabla u_n^\theta|))+c_1(p^*)^{-1}(\overline{m}(c_2|u^\theta_n|))
\end{equation}
with some $k_1 \in E_{m^*}(\Omega)$ and an $N$-function $p$ growing essentially less rapidly than~$\overline{m}$. To get~\eqref{A-ElMes7} it suffices to show
\begin{equation*}\overline{m}^*(|c_1 A_\theta(x,\nabla u_n^\theta)|)\leq \overline{m}(|c_2 \nabla u_n^\theta|),
\end{equation*}
which follows from the Fenchel-Young inequality~\eqref{inq:F-Y},~\eqref{FYeq}, (A3) and $c_A,\theta\in(0,1]$. Indeed,  we have
\begin{multline*}A_\theta(x,\nabla u_n^\theta)\cdot \nabla u_n^\theta\leq \overline{m}\left(\left|\frac{2}{c_A}\nabla u_n^\theta\right|\right)+\overline{m}^*\left(\left|\frac{c_A}{2}A_\theta(x,\nabla u_n^\theta)\right|\right)\\\leq \overline{m}\left(\left|\frac{2}{c_A}\nabla u_n^\theta\right|\right)+c_A \overline{m}^*\left(\left|\frac{1}{2}A_\theta(x,\nabla u_n^\theta)\right|\right),\end{multline*} but on the other hand 
 \begin{equation*} 
 \begin{split}
 A_\theta(x,\nabla u_n^\theta)\cdot \nabla u_n^\theta &\geq c_A M\left(x,\nabla u_n^\theta\right)+c_A M^*\left(x, A(x,\nabla u_n^\theta)\right)+\theta \overline{m}\left(\left| \nabla u_n^\theta\right|\right)+\theta \overline{m}^*\left(\left|\bn m(\nabla u_n^\theta)\right|\right)\\&\geq 2 c_A\left(\frac{1}{2}\overline{m}^*\left(\left|A(x,\nabla u_n^\theta)\right|\right) +\frac{1}{2}\overline{m}^*\left(\left|\theta\bn m(\nabla u_n^\theta)\right|\right) \right) \geq 2 c_A \overline{m}^*\left(\frac{1}{2} \left|A_\theta(x,\nabla u_n^\theta)\right| \right).
 \end{split} 
 \end{equation*}
 Therefore, we get
 \[  c_A \overline{m}^*\left(\frac{1}{2} \left|A_\theta(x,\nabla u_n^\theta)\right| \right)\leq \overline{m}\left(\left|\frac{2}{c_A}\nabla u_n^\theta\right|\right) \]
and further, by convexity of $\overline{m}^*$,
\[ \left|A_\theta(x,\nabla u_n^\theta)\right|  \leq 2(\overline{m}^*)^{-1}\left(\frac{1}{c_A} \overline{m}\left(\left|\frac{2}{c_A}\nabla u_n^\theta\right|\right)\right). \]
Therefore, \cite[Theorem~2]{ElMes} gives the claim.\end{proof}
\subsubsection*{Step 2. Truncated problem}

We prove the existence for non-regularized problem with bounded data by passing to the limit with $\theta\to 0$ in the regularized truncated problem~\eqref{eq:reg-bound}.
\begin{prop}Suppose $A$ and $M$ satisfy conditions (A1)-(A3) and (M).  Let $f\in L^1(\Omega_T)$ and $u_0\in L^1(\Omega)$. Then for every  $n\in\N$ there exists a weak solution to the problem  \begin{equation}\label{eq:bound}
\left\{\begin{array}{ll}
\partial_t u_n -\dv A (x,\nabla u_n )= T_n (f) & \ \mathrm{ in}\  \OT,\\
u_n(t,x)=0 &\ \mathrm{  on} \ (0,T)\times\partial\Omega,\\
u_n (0,\cdot)=u_{0,n}(\cdot)=T_n(u_0)\in L^1(\Omega) & \ \mathrm{ in}\  \Omega.
\end{array}\right.
\end{equation} Namely, there exists $u_n \in \VTMi$, such that for any $\vp\in C_c^\infty([0;T)\times \Omega)$
\begin{eqnarray}\label{weak-reg} 
-\iOT u_n \partial_t \vp dx\, dt- \iO u_n(0)\vp(0)\, dx+\iOT A (x,\nabla u_n )\cdot \nabla \vp \, dx\, dt= \iOT T_n(f)\vp\, dx\, dt.
\end{eqnarray}  
\label{prop:bound}
\end{prop}
\begin{proof} We apply Proposition~\ref{prop:reg-bound} and  let $\theta\to 0$.

\medskip

\textbf{Existence of weak limits.} For this step we need certain a priori estimates. By energy equality~\eqref{en:eq}, (A2) and~\eqref{FYeq} we get
\begin{equation}
\label{en:eq:imp}
\begin{split}&\frac{1}{2}\iO (u_n^\theta(\tau))^2\, dx+
 \int_{\Omega_\tau} c_A M(x,\nabla u_n^\theta)+ c_A M^*(x, A (x,\nabla u_n^\theta))\, dx\, dt+\\
&+\int_{\Omega_\tau}\theta m( \nabla u_n^\theta )+\theta m^*(\bar{\nabla} m(  {\nabla} u_n^\theta ))\,dx\,dt \leq  \int_{\Omega_\tau} T_n(f) u_n^\theta\, dx\, dt+\frac{1}{2} \iO (u_{0,n})^2dx.\end{split}
\end{equation}
To estimate the right-hand side we are going to apply the Fenchel-Young inequality~\eqref{inq:F-Y} and the~modular Poincar\'{e} inequality (Theorem~\ref{theo:Poincare}). For this let us consider an $N$-function $P:[0,\infty)\to[0,\infty)$ satisfying $\Delta_2$-condition and such that $P(s)\leq \frac{c_A}{2c_P}\inf_{x\in\Omega,\,\xi:\,|\xi|=s}M(x,\xi)$, where $c_P$ is the constant from the modular Poincar\'{e} inequality for $P$ (see Lemma~\ref{lem:constr} for a construction). Then on the right-hand side of~\eqref{en:eq:imp} we have
\[\begin{split}
\int_{\Omega_\tau} T_n(f) u_n^\theta\, dx\, dt&\leq \int_{\Omega_\tau} P^*(|T_n(f)|)\, dx\, dt+\int_{\Omega_\tau} P(|u_n^\theta|)\, dx\, dt\\
&\leq \int_{\Omega_\tau} P^*(|T_n(f)|)\, dx\, dt+c_P\int_{\Omega_\tau} P(|\nabla u_n^\theta|)\, dx\, dt\\
&\leq \int_{\Omega_\tau} P^*(|T_n(f)|)\, dx\, dt+\frac{c_A}{2}\int_{\Omega_\tau} M(x,\nabla u_n^\theta )\, dx\, dt.
\end{split}\]
Consequently, we infer that~\eqref{en:eq:imp} implies\begin{equation*}
\begin{split}&\frac{1}{2}\iO (u_n^\theta(\tau))^2\, dx+
\int_{\Omega_\tau} \frac{c_A}{2} M(x,\nabla u_n^\theta)+ c_A M^*(x, A (x,\nabla u_n^\theta))\, dx\, dt+\\
&+\int_{\Omega_\tau}\theta m( \nabla u_n^\theta )+\theta m^*(\bar\nabla m( {\nabla} u_n^\theta ))\,dx\,dt \leq  \int_{\Omega_\tau} P^*(|T_n(f)|)\, dx\, dt+\frac{1}{2} \iO   (u_{0,n})^2dx.\end{split}
\end{equation*}
Note that the right-hand side above is bounded for fixed $n$. Namely,
\begin{equation*}
\int_{\Omega_\tau} P^*(|T_n(f)|)\, dx\, dt+\frac{1}{2} \iO   (u_{0,n})^2dx\leq  \left(P^*(n)\cdot T+\frac{n^2}{2}\right)|\Omega|=:w_1(n).
\end{equation*}
When we take into account that $\tau$ is arbitrary, this observation implies
\begin{eqnarray}
&\sup_{\tau \in[0,T]}\|u_n^\theta(\tau)\|^2_{L^2(\Omega)}&\leq w_1(n),\label{apriori:utL2} \\
&\frac{c_A}{2}\int_{\OT}  M(x,\nabla u_n^\theta)\,dx\,dt&\leq w_1(n),\label{apriori:Mt}\\
&c_A\int_{\OT} M^*(x, A (x,\nabla u_n^\theta))\,dx\,dt&\leq w_1(n),\label{apriori:Mst}
\\
&\int_{\OT}  \theta m^*(\bar\nabla m(|{\nabla} u_n^\theta|))\,dx\,dt&\leq w_1(n).\label{apriori:mst} 
\end{eqnarray}
Therefore, there exist a subsequence of $\theta\to 0$, such that
\begin{eqnarray}
\label{conv:untLi}  u_n^\theta\xrightharpoonup* u_n\quad& \text{weakly-* in }& L^\infty(0,T;L^2(\Omega)),\\
\label{limDut}\nabla u_n^\theta \xrightharpoonup* \nabla u_n\quad& \text{weakly-* in }& L_M(\OT;\rn),
\end{eqnarray}
with some $u_n\in\VTMi$ and there exists $\alpha^n\in L_{M^*}(\OT;\rn)$, such that
\begin{equation}
\label{limaDut}A(\cdot,\nabla u_n^\theta) \xrightharpoonup* \alpha^n\quad \text{weakly-* in } L_{M^*}(\OT;\rn).\end{equation}

\medskip 

\textbf{Identification of the limit $\alpha^n$. Uniform estimates.} We fix arbitrary $n$ and show 
\begin{equation}\label{limsup<aln}\limsup_{\theta\to 0}\iOT A(x,\nabla u_n^\theta)\cdot\nabla u_n^\theta\,dx\,dt\leq \iOT \alpha^n \nabla u_n \,dx\,dt.\end{equation} 
Recall the weak formulation of the regularized problem~\eqref{weak-reg-bound} 
\begin{equation}\label{weak} -\iOT u_n^\theta\partial_t \vp dx\, dt- \iO u_n^\theta(0)\vp(0)\, dx+\iOT A_\theta(x,\nabla u_n^\theta)\cdot \nabla \vp\, dx\,dt= \iOT T_n(f)\vp\, dx\, dt\end{equation} 
holding for $\vp\in C_c^\infty([0,T)\times {\Omega})$, where in the first term on the left-hand side, due to~\eqref{conv:untLi}, we have
\begin{equation*}
\lim_{\theta\searrow 0} \iOT u_n^\theta\partial_t \vp\, dx\, dt=\iOT u_n \partial_t \vp\, dx\, dt.
\end{equation*}
Moreover, we prove that 
\begin{equation*}
\lim_{\theta\searrow 0}\iOT\theta\bar{\nabla}m(\nabla u_n^\theta)\cdot\nabla \vp\,dx\,dt=0.
\end{equation*}
   To get this, we split $\OT$ into
\[\Omega^\theta_{T,R}=\{(t,x)\in \OT:|\nabla u_n^\theta|\leq R\}\]
and its complement and consider the following integrals separately
\[\iOT\theta\bar{\nabla}m(\nabla u_n^\theta)\cdot\nabla \vp\,dx\,dt=\int_{\Omega^\theta_{T,R}}\theta\bar{\nabla}m(\nabla u_n^\theta)\cdot\nabla \vp\,dx\,dt+\int_{\OT\setminus\Omega^\theta_{T,R}}\theta\bar{\nabla}m(\nabla u_n^\theta)\cdot\nabla \vp\,dx\,dt.\]
To deal with the first term on the right-hand side above, we use continuity of $\bar{\nabla} m$ to obtain
\[\lim_{\theta\searrow 0}\int_{\Omega^\theta_{T,R}}\theta\bar{\nabla}m(\nabla u_n^\theta)\cdot\nabla \vp\,dx\,dt\leq \lim_{\theta\searrow 0} \left(\theta|\OT|\cdot\|\nabla \vp\|_{L^\infty(\OT;\rn)}\sup_{\xi:\,|\xi|\leq R}|\bar{\nabla}m(\xi)|\right)=0.\] 

As for the integral over $\OT\setminus\Omega^\theta_{T,R}$, let us notice that due to a priori estimate~\eqref{apriori:Mt}, the sequence $\{\nabla u_n^\theta\}_\theta$ is uniformly bounded in $L^1(\OT)$ and thus
\begin{equation}
\label{qtrc:meas}\sup_{0\leq\theta\leq 1}|\OT\setminus\Omega^\theta_{T,R}|\leq \frac{C}{R}.
\end{equation} 
Furthermore, since $m^*$ is an $N$-function, for $\theta\in(0,1)$ we have $m^*(\theta\cdot)\leq\theta m^*(\cdot).$ This together with  $L^1(\OT)$-bound~\eqref{apriori:mst} for 
$ \theta m^*(\bar\nabla m( {\nabla} u_n^\theta ))$, which is  uniform with respect to $\theta$, we get $L^1(\OT)$-bound  for 
$ \{ m^*(\theta\bar\nabla m( {\nabla} u_n^\theta ))\}_\theta$. Therefore, Lemma~\ref{lem:unif} implies the uniform integrability of $\{\theta  \bar\nabla m( {\nabla} u_n^\theta)\}_\theta$. Then, using~\eqref{qtrc:meas}, we obtain
\[
\int_{\OT\setminus\Omega^\theta_{T,R}}\theta\bar{\nabla}m(\nabla u_n^\theta)\cdot\nabla \vp\,dx\,dt\leq \|\nabla \vp\|_{L^\infty(\OT;\rn)} \sup_{\theta\in (0,1)}\int_{\OT\setminus\Omega^\theta_{T,R}}\theta|\bar{\nabla}m(\nabla u_n^\theta)|\,dx\,dt\xrightarrow[R\to\infty]{} 0.
\]

Therefore,  we can pass to the limit in the weak formulation of the regularized problem~\eqref{weak}. Because of~\eqref{limaDut} 
 we  obtain
\begin{equation} 
\label{lim:theta=0}-\iOT u_n \partial_t \vp dx\, dt- \iO u_{0,n}\vp(0)\, dx+\iOT \alpha^n\cdot \nabla  \vp\, dx= \iOT T_n(f)\vp\, dx\, dt. 
\end{equation}

Considering the energy equality~\eqref{en:eq} in the first term on the left-hand side we take into account the weak lower semi-continuity of $L^2$-norm and~\eqref{conv:untLi} and realize that
\[\|u_n^\theta(\tau)\|^2_{L^2(\Omega)}=\lim_{\epsilon\to 0} \frac{1}{\epsilon}\int_{\tau-\epsilon}^{\tau}\|u_n^\theta(s)\|^2_{L^2(\Omega)}ds\geq \lim_{\epsilon\to 0}\frac{1}{\epsilon} \int_{\tau-\epsilon}^{\tau}\|u_n (s)\|^2_{L^2(\Omega)}ds=\|u_n (\tau)\|^2_{L^2(\Omega)}.\]
When we erase the nonnegative term~$\iOT\theta\bar{\nabla}m(\nabla u_n^\theta)\cdot\nabla u_n^\theta\,dx\,dt$  in~\eqref{en:eq} and then pass to the limit with $\theta\searrow 0$, we get
\begin{equation}\label{en-eq-impl} \frac{1}{2}\|u_n (\tau)\|^2_{L^2(\Omega)}- \frac{1}{2}\|u_{0,n}\|^2_{L^2(\Omega)}+\limsup_{\theta\searrow 0}\iOT A (x,\nabla u_n^\theta)\cdot \nabla u_n^\theta\, dx\, dt\leq \iOT T_n(f) u_n \, dx\, dt.\end{equation}

By Lemma~\ref{lem:intbyparts}  {\em ii)} applied to~\eqref{lim:theta=0}  with $A=\alpha^n$,  $F=T_n(f)$, and $h(\cdot)= T_k(\cdot)$, we obtain 
\[ 
 -\int_{\OT} \left(\int_{u_{0,n}(x)}^{u_n (t,x)} T_k(\s)d\s\right) \partial_t  \xi \,dx\,dt =- \int_{\OT} \alpha^n\cdot \nabla (T_k(u_n)\xi) \,dx\,dt+\int_{\OT}T_n(f) T_k(u_n)\xi \,dx\,dt, 
\]
for every $\xi\in C_c^\infty([0,T)\times \overline{\Omega})$. Then taking~$\xi(t,x)=\vt(t)$, given by~\eqref{vt}, we get
\begin{equation}\begin{split}\label{przedlim}
&-\int_{\OT} \left(\int_{0}^{u_n (t,x)} T_k(\s)d\s-\int_{0}^{u_{0,n}(x)} T_k(\s)d\s\right) \partial_t  \vt    \,dx\,dt\\
&=-\int_{\OT}\alpha^n \cdot \nabla (T_k(u_n))\vt \,dx\,dt+\int_{\OT}T_n(f) T_k(u_n)\vt \,dx\,dt.\end{split}
\end{equation}
On the right-hand side we integrate by parts obtaining\[ -\int_{\OT} \left(\int_{u_{0,n}}^{u_n (t,x)} T_k(\s)d\s \right) \partial_t  \vt    \,dx\,dt=\int_{\OT}\partial_t \left(\int_{u_{0,n}}^{u_n (t,x)} T_k(\s)d\s \right)   \vt \,dx\,dt.\]
Then we pass to the limit with ${r}\to 0$, apply the Fubini theorem, and integrate over the time variable\[\begin{split}\lim_{{r}\to 0} \int_{\OT}\partial_t \left(\int_{u_{0,n}}^{u_n (t,x)} T_k(\s)d\s \right)   \vt \,dx\,dt&=  \int_{\Omega_\tau}\partial_t \left(\int_{u_{0,n}}^{u_n (t,x)} T_k(\s)d\s \right)   \,dx\,dt\\
&=  \int_{\Omega } \left(\int_{u_{0,n}}^{u_n (\tau,x)} T_k(\s)d\s \right)   \,dx.\end{split}\]

Passing with ${r}\to 0$ in~\eqref{przedlim} for a.e.~$\tau\in[0,T)$ we get
\[ 
 \int_{\Omega} \left(\int_{0}^{u_n (\tau,x)} T_k(\s)d\s-\int_{0}^{u_{0,n}(x)} T_k(\s)d\s\right) \,dx  =-\int_{\Omega_\tau}\alpha^n \cdot \nabla  T_k(u_n) \,dx\,dt+\int_{\Omega_\tau}T_n(f) T_k(u_n)  \,dx\,dt. 
\]
Applying the Lebesgue Monotone Convergence Theorem for $k\to\infty$ we obtain
\[\frac{1}{2}\|u_n(\tau)\|^2_{L^2(\Omega)}-\frac{1}{2}\|u_{0,n}\|^2_{L^2(\Omega)}=-\int_{\Omega_\tau}\alpha^n \cdot \nabla u_n \,dx\,dt+\int_{\Omega_\tau}T_n(f) u_n  \,dx\,dt,\]
which combined with~\eqref{en-eq-impl} gives~\eqref{limsup<aln}.

\medskip

\textbf{Identification of the limit $\alpha_n$. Conclusion by monotonicity argument.} Let us recall that $n$ is fixed and concentrate on proving  \begin{equation}
\label{lim=ca}A(x,\nabla u_n)=\alpha^n\qquad \text{a.e.}\quad\text{in}\quad \Omega_T.
\end{equation}
 Monotonicity assumption (A3) of~$A$ implies
\[(A(x,\eta)-A(x,\nabla u_n^\theta))\cdot(\eta-\nabla u_n^\theta)\geq 0\qquad\text{a.e.  in }\OT,\]
for any $\eta\in L^\infty(\OT;\rn)\subset E_M(\OT;\rn)$. Since $A(x,\eta)\in L_{M^*}(\OT,\rn)= (E_{M}(\OT,\rn))^*$, we pass to the limit with $\theta\searrow 0$ and take into account~\eqref{limsup<aln} to conclude that
\begin{equation}
\label{mono:int}
\iOT (A(x,\eta)-\alpha^n)\cdot(\eta-\nabla u_n)\,dx\,dt\geq 0.
\end{equation}

Let us define \begin{equation}
\label{omm}
\Omega_T^K=\{(t,x)\in\OT:\ |\nabla u_n|\leq K\quad\text{a.e. in }\OT\}.
\end{equation} 

We are going to show that
\begin{equation}
\label{lim:id:omj}
A(x,\nabla u_n)=\alpha^n\qquad \text{a.e.}\quad\text{in}\quad \Omega_T^j
\end{equation}
for arbitrary $j$. This implies the equality a.e. in $\Omega_T$, i.e.~\eqref{lim=ca}.

We fix arbitrary $0<j<i$ and $z\in L^\infty(\OT;\rn)$. Consider a parameter $h\in(0,1)$ and choose 
\[\eta=\nabla u_n\mathds{1}_{\Omega_T^i}+
hz\mathds{1}_{\Omega_T^j}.\]
Then in~\eqref{mono:int} we have
\begin{equation*}
 - \int_{\OT\setminus\Omega_T^i}(A(x,0)-\alpha^n) \nabla u_n \,dx\,dt +h\int_{ \Omega_T^j} (A(x,\nabla u_n+hz)-\alpha^n)z \, dx\,dt\geq  0.
\end{equation*} 
which becomes
\begin{equation}
\label{po-mon} 
 \int_{\OT\setminus\Omega_T^i} \alpha^n \cdot \nabla u_n \,dx\,dt +h\int_{ \Omega_T^j} (A(x,\nabla u_n+hz)-\alpha^n)z \, dx\,dt\geq  0,
\end{equation} 
because (A2) implies $A(x,0)=0$. The Fenchel-Young inequality applied to $|\alpha^n \cdot \nabla u_n|$ above ensures (via the Lebesgue Dominated Convergence Theorem) that the first integral on the left-hand side vanishes when $i\to \infty$. Therefore, after passing with $i\to \infty$ in~\eqref{po-mon}, we obtain \begin{equation*}
\int_{ \Omega_T^j} (A(x,\nabla u_n+hz)-\alpha^n)\cdot z\,  dx\, dt\geq  0\qquad\forall_{h\in(0,1)}
\end{equation*}
where $z$ and $j$ are fixed. We are going to pass to the limit with $h\to 0$. 

Note that
\[A(x,\nabla u_n+hz)\xrightarrow[h\to 0]{}A(x,\nabla u_n)\quad\text{a.e. in}\quad \Omega_T^j.\] 
Moreover, as $\{A(x,\nabla u_n+hz)\}_h$ is bounded on $\Omega_T^j$, Lemma~\ref{lem:M*<M}  results in
\[  M^*\left(x,A(x,\nabla u_n+hz)\right) \leq \frac{2}{c_A}  M\left(x,\frac{2}{c_A}(\nabla u_n+hz)\right).\]
The right-hand side is bounded, because $\{ M(x,\frac{2}{c_A}(\nabla u_n+hz))\}_h$ is uniformly bounded in $L^1(\Omega_T^j)$ (cf.~\eqref{LinfinEM} and~\eqref{omm}). Then also $\{M^*\left(x,A(x,\nabla u_n+hz)\right)\}_h$ is uniformly bounded in $L^1(\Omega^j_T)$. Hence, Lemma~\ref{lem:unif} gives uniform integrability of $\{A(x,\nabla u_n+hz)\}_h$ on $\Omega_T^j$. When we notice that $|\Omega_T^j|<\infty$, we can apply the Vitali Convergence Theorem (Theorem~\ref{theo:VitConv}) to get
\[A(x,\nabla u_n+hz)\xrightarrow[h\to 0]{}A(x,\nabla u_n)\quad\text{in}\quad L^1(\Omega_T^j;\rn).\] 
Thus
\begin{equation*} \int_{ \Omega_T^j} (A(x,\nabla u_n+hz)-\alpha^n)\,z\, dx\, dt\xrightarrow[h\to 0]{} \int_{ \Omega_T^j} (A(x,u_n)-\alpha^n)\,z \, dx\,dt.
\end{equation*}
Consequently,
\begin{equation*}  \int_{ \Omega_T^j} (A(x,\nabla u_n)-\alpha^n)\,z\,  dx\, dt\geq 0,
\end{equation*}
for any $z\in L^\infty(\OT;\rn)$. Let us take 
\[z=\left\{\begin{array}{ll}-\frac{A(x,\nabla u_n)-\alpha^n}{|A(x,\nabla u_n)-\alpha^n|}&\ \text{if}\quad A(x,\nabla u_n)-\alpha^n\neq 0,\\
0&\ \text{if}\quad A(x,\nabla u_n)-\alpha^n= 0.
\end{array}\right.\]
We obtain 
\begin{equation*}  \int_{ \Omega_T^j} |A(x,\nabla u_n)-\alpha^n| dx\,dt\leq 0,
\end{equation*}
hence~\eqref{lim:id:omj} holds. Consequently, we get \eqref{lim=ca}, which completes the proof.

\medskip

\textbf{Conclusion of the proof of Proposition~\ref{prop:bound}.} We pass to the limit in  the weak formulation of bounded regularized problem~\eqref{weak} due to~\eqref{conv:untLi}, \eqref{limDut}, \eqref{limaDut}, and~\eqref{lim=ca}, getting the existence of~$u_n\in\VTMi$ satisfying
\[
-\iOT u_n \partial_t \vp dx\, dt- \iO u_n(0)\vp(0)\, dx+\iOT A (x,\nabla u_n )\cdot \nabla \vp \, dx\, dt= \iOT T_n(f)\vp\, dx\, dt\quad\forall_{\vp\in C_c^\infty([0;T)\times \Omega)},\]
i.e.~\eqref{weak-reg}, which ends the proof. \end{proof}

\subsubsection*{Step 3. Convergence of truncations $T_k (u_n)$}
\begin{prop}\label{prop:convTk} Suppose $A$ and $M$ satisfy conditions (A1)-(A3) and (M). 
 Let $f\in L^1(\Omega_T)$, $u_0=u(x,0)\in L^1(\Omega)$, and $u_n \in\VTMi$ denote a weak solution to the problem~\eqref{eq:bound}. Let $k>0$ be arbitrary.

Then there exists $u\in\VTM$ such that, up to a subsequence, we have
\begin{eqnarray} T_k(u_n) &\xrightharpoonup{} &  T_k(u)\quad \text{ in } L^1(0,T;W^{1,1}_0(\Omega)),\label{conv:n:TkuL1}\\ 
T_k(u_n) &\xrightharpoonup{*} &  T_k(u)\quad \text{weakly-* in } L^\infty (\OT),\label{conv:n:TkuLi}\\
\nabla T_k(u_n) &\xrightharpoonup* &\nabla T_k(u)\quad \text{weakly-* in } L_M(\OT;\rn),\label{conv:n:TkutM}\\
 A(x,\nabla T_k(u_n)) &\xrightharpoonup*& \Ak \quad \text{weakly-* in } L_{M^*}(\OT;\rn),\label{conv:n:ATkutMs}
\end{eqnarray}
for some $\Ak\in L_{M^*}(\OT;\rn)$.
\end{prop} 
\begin{proof} We apply Lemma~\ref{lem:intbyparts}  {\em ii)} to $u_n^\theta$ solving~\eqref{weak-reg-bound},  for which we know $T_k(u_n^\theta)\in V_T^M$ and $u_n^\theta(t,x)\in L^\infty([0,T];L^1(\Omega))$. Consider Lemma~\ref{lem:intbyparts}  with $A=A_\theta(x,\nabla u_n^\theta)$, $F=T_n(f)$, $h(\cdot)=T_k(\cdot)$, and $\xi(t,x)=\vt(t),$ defined in~\eqref{vt}, we obtain
\[\begin{split}
&-\int_{\OT} \left(\int_{u_{0,n}(x)}^{u_n^\theta(t,x)} T_k(\s)d\s\right) \partial_t (\vt)   \,dx\,dt+\int_{\OT}A_\theta(x,\nabla u_n^\theta)\cdot \nabla (T_k(u_n^\theta)\vt) \,dx\,dt\\
&=\int_{\OT}T_n(f) T_k(u_n^\theta)\vt \,dx\,dt.\end{split}
\] 

When we pass to the limit with ${r}\to 0$ (cf.~\eqref{przedlim}), for a.e. $\tau\in [0,T]$ we get
\[\begin{split}
&\int_{\Omega } \left(\int_{0}^{u_n^\theta(\tau,x)} T_k(\s)d\s-\int_{0}^{u_{0,n}(x)} T_k(\s)d\s\right)  \,dx +\int_{\Omega_\tau}A_\theta(x,\nabla u_n^\theta) \cdot \nabla  T_k(u^\theta_n) \,dx\,dt\\
&= \int_{\Omega_\tau}T_n(f) T_k(u^\theta_n)  \,dx\,dt,\end{split}
\]
and consequently
\[\begin{split}
&\frac{1}{2}\|T_k \left( u_n^\theta(\tau)\right)\|^2_{L^2(\Omega)} -\frac{1}{2}\|T_k \left( u_{0,n}\right)\|^2_{L^2(\Omega)}+\int_{\Omega_\tau}A(x,\nabla T_k (u_n^\theta)) \cdot \nabla T_k(u_n^\theta) \,dx\,dt\\
&+\int_{\Omega_\tau}\theta \bar{\nabla}m (\nabla u_n^\theta) \cdot \nabla T_k(u_n^\theta) \,dx\,dt= \int_{\Omega_\tau}T_n(f) T_k(u^\theta_n)\,dx\,dt.\end{split}\]
Applying further~\eqref{FYeq} and (A2) we get
\[\begin{split}
&\frac{1}{2}\|T_k \left( u_n^\theta(\tau)\right)\|^2_{L^2(\Omega)} -\frac{1}{2}\|T_k \left( u_{0,n}\right)\|^2_{L^2(\Omega)}+\int_{\Omega_\tau} c_A M(x,\nabla  T_k (u_n^\theta))+ c_A M^*(x, A (x, \nabla T_k (u_n^\theta)))\,dx\,dt\\
&+\int_{\Omega_\tau}\theta m(|\nabla  T_k (u_n^\theta)|)+\theta m^*(\bar{\nabla} m(| {\nabla}  T_k (u_n^\theta)|))\,dx\,dt\leq \int_{\Omega_\tau}T_n(f) T_k(u^\theta_n)\,dx\,dt\leq k\|f\|_{L^1(\OT)}  .\end{split}\]
When we notice that \[\frac{1}{2}\|T_k \left( u_{0,n}\right)\|^2_{L^2(\Omega)}\leq \frac{k}{2}\|  u_{0,n} \|_{L^1(\Omega)} = \frac{k}{2}\|  T_n(u_{0}) \|_{L^1(\Omega)}   \] and since $\tau\in( 0,T)$ is arbitrary, for \[w_2(k):= k\left(\|f\|_{L^1(\OT)}+\frac{1}{2}\|  u_{ 0} \|_{L^1(\Omega)}\right),\] we obtain
\begin{eqnarray}
&c_A\int_{\Omega_\tau}  M(x,\nabla  T_k (u_n^\theta)) \,dx\,dt\leq w_2(k),\nonumber\\
&c_A\int_{\Omega_\tau}  M^*(x, A (x, \nabla T_k (u_n^\theta)))\,dx\,dt\leq w_2(k),\label{apriori:thetai}\\
&\int_{\Omega_\tau}\theta m(|\nabla  T_k (u_n^\theta)|) \,dx\,dt\leq w_2(k).\nonumber
\end{eqnarray}
Then, due to~\eqref{LminEM}, for each fixed~$\theta\in(0,1)$\begin{equation*}
\nabla T_k(u_n^\theta)\in L_m(\OT;\rn)\subset E_M(\OT;\rn).
\end{equation*} 
Moreover, a priori estimates~\eqref{apriori:utL2}, 
\eqref{apriori:Mt}, \eqref{apriori:Mst}, \eqref{apriori:mst}, and \eqref{apriori:thetai} and weak lower semi-continuity of~a~convex functional together imply existence of $u\in \VTM$ such that \eqref{conv:n:TkuL1},\eqref{conv:n:TkuLi},\eqref{conv:n:TkutM} hold, and existence of $\Ak$ such  that~\eqref{conv:n:ATkutMs} holds.\end{proof}

\subsubsection*{Step 4. Controlled radiation }
 
 \begin{prop}\label{prop:contr:rad:n}
 Suppose $A$ and $M$ satisfy conditions (A1)-(A3) and (M), and $f\in L^1(\Omega)$. Assume further that $u_n$ is a weak solution to~\eqref{eq:bound}, $n>0$. Then
\begin{equation}
\lim_{l\to\infty} |\{|u_n|>l\}|=0.\label{conv:umeas:n}
\end{equation} 
and \begin{equation}
 \label{eq:contr:rad:n}
 \lim_{l\to\infty}\limsup_n\int_{\{l<|u_n|<l+1\}} A(x,\nabla u_n)\nabla u_n\,dx\,dt=0.
 \end{equation}
\end{prop}

\begin{proof}  To prove~\eqref{conv:umeas:n} we note that if $\dm$ is an $N$-function satisfying $\Delta_2$-condition such that $\dm(s)\leq \inf_{x\in\Omega,\,\xi:|\xi|=s} M(x,\xi)$ (see Lemma~\ref{lem:constr} for a construction)  for we have\[|\{|u_n|\geq l\}|=|\{|T_l(u_n)|= l\}|=|\{|T_l(u_n)|\geq l\}|=|\{\dm(|T_l(u_n)|)\geq \dm(l)\}|.\]
Moreover, for $l>0$ we have 
\begin{equation*}
\begin{split}
|\{|u_n|\geq l\}|&\leq \int_{\OT} \frac{\dm(|T_l(u_n)|)}{\dm(l)} dx\,dt\leq \frac{c(N,\Omega,T)}{\dm(l)} \int_\OT\dm( |\nabla T_l(u_n)|)dx\,dt\leq \\
&\leq  \frac{c(N,\Omega,T)}{\dm(l)} \int_\OT M(x,  \nabla T_l(u_n) )dx\,dt\leq \\&
\leq \frac{C(M,N,\Omega,T)}{\dm(l)}  \cdot l \left(\|f\|_{L^1(\OT)}+\frac{1}{2}\|  u_{ 0} \|_{L^1(\Omega)}\right) \leq \\& \leq C(f,u_0,M,N,\Omega,T)\frac{l}{\dm(l)}  \xrightarrow[ l\to \infty]{}0 
.\end{split}\end{equation*}
In the above estimates we apply (respectively) the Chebyshev inequality, the Poincar\'{e} inequality (Theorem~\ref{theo:Poincare}), a priori estimate~\eqref{apriori:thetai} and the facts that $f,u_0\in L^1(\Omega)$ and that $\dm$ is an $N$-function (cf.~Definition~\ref{def:Nf}).

To prove~\eqref{eq:contr:rad:n}, we consider nonincreasing functions $\phi_{r}\in C_c^\infty([0,T))$, given by~\eqref{phidelta}, and 
\[G_l(s):=T_{l+1}(s)-T_{l}(s).\]

Since $u_n\in \VTMi$ is a weak solution to~\eqref{eq:bound}, we can use $\vp(t,x)=G_l(u_n(t,x))\phi_{r}(t)$ as a test function and obtain
\begin{equation*}
\iOT (\partial_t u_n) G_l(u_n) \phi_{r} \, dx\,dt+\int_{\{l<|u_n|<l+1\}} A(x,\nabla u_n)\nabla u_n\,\phi_{r}\, dx\,dt=\iOT T_n(f)G_l(u_n) \phi_{r} \, dx\,dt.
\end{equation*}
Notice that on the left-hand side above we have
\[\begin{split}
\int_0^T (\partial_t u_n) G_l(u_n) \phi_{r} \,dt&=
\int_0^T \partial_t \left(\int_0^{u_n} G_l(s)ds\right)\, \phi_{r} \,dt\\
&=
-\phi_{r}(0)\int_0^{u_{0,n}}  G_l(s)ds -
\int_0^T  \int_0^{u_n} G_l(s)ds \,\partial_t \phi_{r} \,dt.\end{split}\]
Moreover, $\int_0^{u_n} G_l(s)ds\geq 0$ and $\partial_t \phi_{r}\leq 0$, hence \[-\int_0^T  \int_0^{u_n} G_l(s)ds \,\partial_t \phi_{r} \,dt\geq 0\]
and consequently
\[ \int_{\{l<|u_n|<l+1\}} A(x,\nabla u_n)\nabla u_n\,\phi_{r}\, dx\,dt\leq \iOT T_n(f)G_l(u_n) \phi_{r} \, dx\,dt+\iO \phi_{r}(0)\int_0^{u_{0,n}}  G_l(s)ds\,dx.\]

Furthermore, to infer that the right-hand side above tends to zero when $l\to\infty$, it suffices to~observe that
\[\iO\int_0^{|u_{0,n}|} |G_l(s)|\,ds\,dx\leq \iO\int_0^{|u_{0,n}|} \mathds{1}_{\{s>l\}}\,ds\,dx=\int_{\{||u_{0,n}|-l|>0\}} \left(|u_{0,n}|-l\right) dx \xrightarrow[l\to\infty]{}0\]
and
\[\iOT T_n(f)G_l(u_n) \phi_{r} \, dx\,dt\leq \int_{\{|u_n|>l\}} |f|\, dx\,dt\xrightarrow[l\to\infty]{}0.\] 
Therefore,~\eqref{eq:contr:rad:n} follows.\end{proof}

\subsubsection*{Step 5. Almost everywhere limit}

 \begin{prop}\label{prop:contr:rad}
 Suppose $A$ and $M$ satisfy conditions (A1)-(A3) and (M), and $f\in L^1(\Omega)$. Assume further that $u_n$ is a weak solution to~\eqref{eq:bound}, $n>0$.  For the  function $u \in \VTM $  coming from Proposition~\ref{prop:convTk} we have
\begin{equation}
 u_n\to u \quad  a.e.\ \text{in}\ \OT,\label{conv:usae}\end{equation}
and
\begin{equation}
\lim_{l\to\infty} |\{|u|>l\}|=0.\label{conv:umeas}
\end{equation} 
\end{prop}

\begin{proof}  To prove~\eqref{conv:umeas} we apply the comparison principle (Proposition~\ref{prop:comp-princ}). We can do it since weak solutions $u_n$ are renormalized ones.

We define asymmetric truncations as follows\[T^{k,l}(f)(x)=\left\{\begin{array}{ll}
-k& f\leq -k,\\
f & |f|\leq k,\\
l&  f \geq l.
\end{array}\right.\]
 Let $u^{a,b}$ denote  a weak solution to \[u_t-{\dv }A(x,\nabla u )= T^{a,b}(f),\qquad u(0,x)=T^{a,b}(u_0),\] which exists according to Proposition~\ref{prop:bound}. 

When $0<l<l'$ and $0<k<k'$, Proposition~\ref{prop:comp-princ} implies that
\begin{equation}
\label{order}
u^{k',l}\leq u^{k,l}\leq u^{k,l'}
\end{equation}
for a.e. $(t,x)\in\OT$. Due to the monotonicity of $(u^{k,l})_l$ we deduce that $\lim_{l\to\infty} u^{k,l}$ exists a.e. in~$\OT$. Let us denote it by $u^{k,\infty}$. On the other hand, taking into account~\eqref{order} we infer that $u^{k',\infty}\leq u^{k,\infty}$ a.e. in~$\OT$. Thus,  there exists the limit $u^{\infty,\infty}=\lim_{k\to\infty} u^{k,\infty}$ a.e. in~$\OT$. Consequently, due to~the~uniqueness of the limit (cf.~\eqref{conv:n:TkuL1}), we get the convergence~\eqref{conv:usae}.

Having~\eqref{conv:usae},~\eqref{conv:umeas} is a direct consequence of~\eqref{conv:umeas:n}.\end{proof}

\subsubsection*{Step 6. Identification of the limit of $A(x,\nabla T_k(u_n))$}

In this step we employ the time regularization~\eqref{gmu} and monotonicity trick to identify the limit~\eqref{conv:n:ATkutMs}.

\begin{prop} \label{prop:convTkII} Suppose $A$ and $M$ satisfy conditions (A1)-(A3) and (M). Suppose $u_n$ is a weak solution to~\eqref{eq:reg-bound}, $k>0$ is arbitrary, and $f\in L^1(\Omega)$.  We have\begin{equation}
A(x,\nabla T_k(u_n)) \xrightharpoonup* A(x,\nabla T_k(u))\quad \text{weakly-* in } L_{M^*}(\OT;\rn).\label{conv:n:AT:id}
\end{equation}
\end{prop}
\begin{proof}
 We shall show, still for fixed $k$, that  in~\eqref{conv:n:ATkutMs} \begin{equation}
\label{Ak=ATk}\Ak=A(x,\nabla T_k(u)).
\end{equation} 

Fix nonnegative ${w}\in  C_c^\infty([0,T)).$ We show now that \begin{equation}
\label{limsup<}
\limsup_{n\to\infty} \iOT {w} A(x,\nabla T_k(u_n))\cdot\nabla(T_k(u_n))\,dx\,dt\leq \iOT {w} \Ak \cdot\nabla(T_k(u))\,dx\,dt
\end{equation}
and then conclude~\eqref{Ak=ATk} via the monotonicity argument.

For any $\mu >0$ we apply regularization~\eqref{gmu} to \[g^k(t,x)=T_k(u(t,x))\mathds{1}_{(0,T)}(t)+\omega_0^k(x) \mathds{1}_{(-\infty,0]}(t).\]
We notice that $(g^k)_\mu \in\VTMi$ and due to~\eqref{gmu:ode}  we have
\begin{equation}
\label{mu:pde}
\begin{split}\partial_t (T_k(u)\mathds{1}_{(0,T)})_\mu+ \mu((T_k(u)\mathds{1}_{(0,T)})_\mu-T_k(u)) =0&\quad \text{ for all } t\in(0,T]\text{ and a.e. } x\in\Omega,\\
(T_k(u))_\mu(0,x) =(g^{k})_\mu(0,x)&\quad\text{ a.e. in }\Omega.\end{split}
\end{equation} Moreover, $(g^{k})_\mu \in W^{1,1}_0(\Omega)\cap L^\infty(\Omega)$, $\nabla (g^{k})_\mu\in L_M(\Omega)$, $\|(g^{k})_\mu\|_{L^\infty(\Omega)}\leq k$ for all $\mu>0,$ $\frac{1}{\mu}\|(g^{k})_\mu\|_{W^{1,1}_0(\Omega)}\to 0$ as $\mu\to \infty$, and $(g^{ k})_\mu\to T_k(u_0)$ a.e. in~$\Omega$  as $\mu\to \infty$. 

In turn, $(T_k(u)\mathds{1}_{(0,T)})_\mu$ is differentiable for a.e. $t\in(0,T)$ and $\partial_t (T_k(u)\mathds{1}_{(0,T)})_\mu\in \VTMi$. Moreover, $\nabla(T_k(u)\mathds{1}_{(0,T)})_\mu=(\nabla T_k(u)\mathds{1}_{(0,T)})_\mu$ for all $t\in(0,T]$ and for a.e. $x\in\Omega$, $\|(T_k(u)\mathds{1}_{(0,T)})_\mu\|_{L^\infty(\OT)}\leq k$ for all $\mu >0$, $(T_k(u)\mathds{1}_{(0,T)})_\mu\to T_k(u)$ a.e. in $\OT$ and weakly-$*$ in $L^\infty(\OT)$, 
\begin{equation}
\label{TkmuinEm}
\nabla (T_k(u)\mathds{1}_{(0,T)})_\mu\in E_M(\OT;\rn).
\end{equation} Due to the Jensen inequality and ~\eqref{conv:usae}, we infer that $(\nabla T_k(u)\mathds{1}_{(0,T)})_\mu\xrightarrow{M} \nabla T_k(u)$ converges modularly in $L_M(\OT;\rn)$. Then also
\begin{equation}
\label{conv:Tkumu}
(T_k(u)\mathds{1}_{(0,T)})_\mu\xrightarrow[\mu\to \infty]{} T_k (u)\qquad\text{strongly in}\quad W^{1,1}(\OT).
\end{equation}

 Recall $\psi_l$ given by~\eqref{psil}. We apply Lemma~\ref{lem:intbyparts} to~\eqref{eq:bound}, i.e. with $A=A(x,\nabla u_n)\in L_M(\OT;\rn)$, $F=T_n f\in L^1(\OT) $  twice: first time with $h(\cdot)=\psi_l(\cdot)T_k(\cdot)$ and $\xi={w}$ and second time with  $h(\cdot)=\psi_l(\cdot)$ and $\xi={w}(T_k(u)\mathds{1}_{(0,T)})_\mu.$ Subtracting the second from the first we get\begin{equation}
\label{1stsum:id:Ak}I_1^{n,\mu,l}+I_2^{n,\mu,l}+I_3^{n,\mu,l}=I_4^{n,\mu,l},
\end{equation}
where
\begin{eqnarray*}
I_1^{n, \mu,l}&=&-\iOT\partial_t{w}\int_{u_{0,n}}^{u_n}\psi_l(s)T_k(s)ds\,dx\,dt
+\iOT\partial_t({w} (T_k(u)\mathds{1}_{(0,T)})_\mu)\int_{u_{0,n}}^{u_n}\psi_l(s)ds\,dx\,dt,\\
I_2^{n,\mu,l}&=& \iOT{w} \psi_l(u_n)   A(x,\nabla  u_n)\cdot \nabla(T_k(u_n)-(T_k(u)\mathds{1}_{(0,T)})_\mu) \,dx\,dt,\\
I_3^{n,\mu,l}&=& \iOT{w} \psi_l'(u_n) (T_k(u_n)-(T_k(u)\mathds{1}_{(0,T)})_\mu) A(x,\nabla  u_n)\cdot \nabla u_n \,dx\,dt,\\
I_4^{n,\mu,l}&= &\iOT{w} T_nf \psi_l(u_n) (T_k(u_n)-(T_k(u)\mathds{1}_{(0,T)})_\mu)  \,dx\,dt.
\end{eqnarray*}
We are going to pass to the limit with $n\to\infty$, then $\mu\to\infty$ and finally with $l\to\infty$. Roughly speaking we show that the limit of $I_1^{n,\mu,l}$ is nonnegative, then let $I_3^{n,\mu,l}$ and $I_4^{n,\mu,l}$ to zero. In turn, we get that the limit of $I_2^{n,\mu,l}$ is nonpositive.

\medskip

\textbf{Limit of $I_1^{n,\mu,l}$. } We are going to prove that\begin{equation}\label{limI1>0}
\limsup_{l\to\infty}\limsup_{\mu\to\infty}\limsup_{n\to\infty} I_1^{n,\mu,l}\geq 0.
\end{equation}

Let us consider\begin{equation*}
 I_1^{n,\mu,l}=I_{1,1}^{n,\mu,l}+I_{1,2}^{n,\mu,l}+I_{1,3}^{n,\mu,l},
\end{equation*}
where 
\begin{equation*}
\begin{split}
I_{1,1}^{n,\mu,l}&=-\iOT\partial_t{w}\int_{u_{0,n}}^{u_n}\psi_l(s)T_k(s)ds\,dx\,dt,\\
I_{1,2}^{n,\mu,l}&=\iOT\partial_t w(T_k(u)\mathds{1}_{(0,T)})_\mu \left(\int_{u_{0,n}}^{u_n}\psi_l(s)ds \right)dx\,dt,\\
I_{1,3}^{n,\mu,l}&=\iOT  w\partial_t  (T_k(u)\mathds{1}_{(0,T)})_\mu \left(\int_{u_{0,n}}^{u_n}\psi_l(s)ds \right)dx\,dt. 
\end{split}
\end{equation*}

To deal with $\lim_{n\to\infty}I_{1,1}^{n,\mu,l}$, notice that 
\begin{equation*}
\begin{split} \int_0^{v}\psi_l(s) T_k(s)ds&=\int_0^v\int_0^{T_k(s)}\psi_l(s) \,d\sigma\,ds=\int_0^v\psi_l(s)\int_0^{T_k(v)}\,d\sigma\,ds
-\int_0^{T_k(v)}\int_0^\sigma \psi_l(s)\,ds\,d\sigma\\
&={T_k(v)}\int_0^v\psi_l(s)  \,ds-\int_0^{T_k(v)}\int_0^\sigma \psi_l(s)\,ds\,d\sigma .\end{split}
\end{equation*}

Therefore
\begin{equation*}
\begin{split}
I_{1,1}^{n,\mu,l}&= -\iOT\partial_t w\left(T_k(u_n)\int_{0}^{u_n}\psi_l(s)ds-
\int_0^{T_k(u_n)}\int_{0}^{\sigma}\psi_l(s)ds\,d\sigma\right)dx\,dt\\
& \quad - \iOT \partial_t  w \left(\int_{0}^{T_k(u_{0,n})}\int_0^\sigma\psi_l(s)ds\,d\sigma-
 {T_k(u_{0,n})}\int_{0}^{u_{0,n}}\psi_l(s)ds \right) dx\,dt\\
 &=  -\iOT\partial_t w\left(T_k(u_n)\int_{0}^{u_n}\psi_l(s)ds-
\int_0^{T_k(u_n)}\int_{0}^{\sigma}\psi_l(s)ds\,d\sigma\right)dx\,dt\\
& \quad + \iO w(0)\left(\int_{0}^{T_k(u_{0,n})}\int_0^\sigma\psi_l(s)ds\,d\sigma-
 {T_k(u_{0,n})}\int_{0}^{u_{0,n}}\psi_l(s)ds \right)dx\\
 &\xrightarrow[n\to\infty]{} -\iOT\partial_t w\left(T_k(u)\int_{0}^{u}\psi_l(s)ds-
\int_0^{T_k(u)}\int_{0}^{\sigma}\psi_l(s)ds\,d\sigma\right)dx\,dt\\
& \quad + \iO w(0)\left(\int_{0}^{T_k(u_{0})}\int_0^\sigma\psi_l(s)ds\,d\sigma-
 {T_k(u_{0})}\int_{0}^{u_{0}}\psi_l(s)ds \right)dx=I_{1,1}^{l},
\end{split}
\end{equation*}
where continuity of the integral and of truncation $T_k(\cdot)$ together with the Lebesgue Dominated Convergence Theorem justify passing to the limit with $n\to\infty$.

In the case of $I_{1,2}^{n,\mu,l}$, according to the pointwise convergence of the integrand when $n\to\infty$ and
\[|I_{1,2}^{n,\mu,l}|\leq \iOT 2 \|\partial_t w\|_{L^\infty}\cdot k\cdot (l+1)\,dx\,dt,\]
 the Lebesgue Dominated Convergence Theorem justifies passing to the limit with $n\to\infty$. When we additionally take into account~\eqref{conv:Tkumu} we pass with $\mu\to\infty$ to get
\[\lim_{\mu\to\infty}\lim_{n\to\infty}I_{1,2}^{n,\mu,l} =\iOT\partial_t w\,T_k(u ) \left(\int_{u_0}^{u}\psi_l(s)ds \right)dx\,dt=I_{1,2}^l.\]

As for $I_{1,3}^{n,\mu,l}$, recalling~\eqref{mu:pde}, we notice that
\[ I_{1,3}^{n,\mu,l}=\iOT  w\mu (T_k(u)-(T_k(u)\mathds{1}_{(0,T)})_\mu) \left(\int_{u_{0,n}}^{u_n}\psi_l(s)ds \right)dx\,dt, \] where we let $n\to\infty$ similarly to the case of $I_{1,2}^{n,\mu,l}$ and obtain
\[\lim_{n\to\infty}I_{1,3}^{n,\mu,l} =\iOT  w\mu (T_k(u )-(T_k(u )\mathds{1}_{(0,T)})_\mu) \left(\int_{u_{0}}^{u}\psi_l(s)ds \right)dx\,dt= I_{1,3,1}^{ \mu,l}+I_{1,3,2}^{ \mu,l}+I_{1,3,2}^{ \mu,l}+I_{1,3,4}^{ \mu,l}\]
with
\begin{equation*}
\begin{split}
I_{1,3,1}^{ \mu,l}&=-\iOT   w\partial_t  (T_k(u)\mathds{1}_{(0,T)})_\mu 
\int_0^{u_0} \psi_l(s)ds dx\,dt \\
&\xrightarrow[\mu\to\infty]{} \iOT \partial_t w{ T_k(u) }\int_0^{u_0} \psi_l(s)ds   \,dx\,dt
+\iO w (0)   T_k(u_0)  \int_0^{u_0} \psi_l(s)ds \, dx = I^l_{1,3,1},\\
I_{1,3,2}^{ \mu,l}&=\iOT  w\mu (T_k(u )-(T_k(u )\mathds{1}_{(0,T)})_\mu)  
\left(\int_0^{u} \psi_l(s)ds-\int_0^{T_k(u)} \psi_l(s)ds\right) dx\,dt=\\
&=\int_{\{|u|>k\}}  w\mu (k{\rm sign}(u)-(T_k(u )\mathds{1}_{(0,T)})_\mu)  
\left(\int_0^{u} \psi_l(s)ds-\int_0^{k{\rm sign}(u)} \psi_l(s)ds\right) dx\,dt,\\
I_{1,3,3}^{ \mu,l}&=\iOT  w\mu (T_k(u )-(T_k(u )\mathds{1}_{(0,T)})_\mu)  
\left(\int_0^{T_k(u)} \psi_l(s)ds-\int_0^{(T_k(u)\mathds{1}_{(0,T)})_\mu} \psi_l(s)ds\right) dx\,dt,\\
I_{1,3,4}^{ \mu,l}&=\iOT  w\partial_t( (T_k(u )\mathds{1}_{(0,T)})_\mu )  
\left(\int_0^{(T_k(u)\mathds{1}_{(0,T)})_\mu} \psi_l(s)ds \right) dx\,dt\\ 
&=-\iOT \partial_t w \int_0^{(T_k(u)\mathds{1}_{(0,T)})_\mu}\int_0^\sigma \psi_l(s)dsd\sigma \,dx\,dt 
-\iO w(0)  \int_0^{\omega_0^{\mu,k}}\int_0^\sigma \psi_l(s)dsd\sigma\, dx \\
&\xrightarrow[\mu\to\infty]{} -\iOT \partial_t w  \int_0^{ T_k(u) }\int_0^\sigma \psi_l(s)dsd\sigma  \,dx\,dt
-\iO w(0)  \int_0^{ T_k(u_0)  }\int_0^\sigma \psi_l(s)dsd\sigma\, dx =I^l_{1,3,4}.
\end{split}
\end{equation*}
The limits follows from the Lebesgue Dominated Convergence Theorem. Moreover, due to the fact that $|(T_k(u)\mathds{1}_{(0,T)})_\mu|\leq k$ a.e. in $\OT$ and the monotonicity of $\sigma\mapsto\int_0^{\sigma} \psi_l(s)ds$ a.e. in $\Omega$, we notice that $I_{1,3,2}^{ \mu,l}, I_{1,3,3}^{ \mu,l}\geq 0$. Therefore,
\begin{equation*} \limsup_{\mu\to\infty}\limsup_{n\to\infty}I_1^{n,\mu,l}\geq I_{1,1}^{l}+I_{1,2}^{l}+I_{1,3,1}^l+I_{1,3,4}^l,
\end{equation*}
consequently
\begin{equation*}\begin{split}
\limsup_{\mu\to\infty}\limsup_{n\to\infty}I_1^{n,\mu,l}&\geq -\iOT\partial_t w\left(T_k(u)\int_{0}^{u}\psi_l(s)ds-
\int_0^{T_k(u)}\int_{0}^{\sigma}\psi_l(s)ds\,d\sigma\right)dx\,dt\\
  &+ \iO w(0)\left(\int_{0}^{T_k(u_{0})}\int_0^\sigma\psi_l(s)ds\,d\sigma-
 {T_k(u_{0})}\int_{0}^{u_{0}}\psi_l(s)ds \right)dx\\
 &+\iOT\partial_t w\,T_k(u ) \left(\int_{u_0}^{u}\psi_l(s)ds \right)dx\,dt\\
 &+\iOT \partial_t w{ T_k(u) }\int_0^{u_0} \psi_l(s)dsd\sigma  \,dx\,dt
+\iO w (0)   T_k(u_0)  \int_0^{u_0} \psi_l(s)dsd\sigma\, dx \\
 &-\iOT \partial_t w  \int_0^{ T_k(u) }\int_0^\sigma \psi_l(s)dsd\sigma  \,dx\,dt
-\iO w(0)  \int_0^{ T_k(u_0)  }\int_0^\sigma \psi_l(s)dsd\sigma\, dx =0,\end{split}
\end{equation*}
which implies~\eqref{limI1>0}.

\medskip
 
\textbf{Limit of $I_3^{n,\mu,l}$.} Since (A2) forces nonnegativeness of $A(x,\nabla  u_n)\cdot \nabla u_n$,  the radiation control~\eqref{eq:contr:rad:n} is equivalent to
\[
 \lim_{l\to\infty}\limsup_n\int_{\{l<|u_n|<l+1\}} |A(x,\nabla u_n)\nabla u_n|\,dx\,dt=0.
 \]

Then 
\[\begin{split}
|I_3^{n,\mu,l}|&=\left| \iOT{w} \psi_l'(u_n) (T_k(u_n)-(T_k(u)\mathds{1}_{(0,T)})_\mu) A(x,\nabla  u_n)\cdot \nabla u_n \,dx\,dt\right|\leq\\
&\leq 2 k||{w}||_{L^\infty(\R)} \int_{\{l<|u_n|<l+1\}} \left| A(x,\nabla  u_n)\cdot \nabla u_n \right|\,dx\,dt,
\end{split}\]
which is independent of $\mu$, so it implies\begin{equation}
\label{lim:I3:idAk}
\lim_{l\to\infty}\lim_{\mu\to\infty}\limsup_{n\to\infty} I_3^{n,\mu,l}=0.
\end{equation}

\medskip

\textbf{Limit of $I_4^{n,\mu,l}$.} To deal with the limit with $n\to\infty$ we apply the Lebesgue Dominated Convergence Theorem due to the continuity of the integrand and~\eqref{conv:usae}, i.e.  $u_n\to u$ a.e. in $\OT$. Moreover, we know that  $(T_k(u)\mathds{1}_{(0,T)})_\mu\to T_k(u)$ a.e. in $\OT$ when $\mu\to\infty$. Boundedness in $L^1$ of the rest terms enables to apply the Lebesgue Dominated Convergence Theorem and pass to the limit with $\mu\to\infty$.  to get\begin{equation}
\label{lim:I4:idAk}
\lim_{l\to\infty} \lim_{\mu\to\infty} \lim_{n\to\infty} I_4^{n,\mu,l}=0.
\end{equation} 

\medskip

\textbf{Identification of the limit $A(x,\nabla(T_k(u_n)))$. Conclusion via the monotonicity trick. }  Recall~\eqref{1stsum:id:Ak}. Passing there to the limit, due to~\eqref{lim:I4:idAk} and \eqref{lim:I3:idAk},  we get 
\[\begin{split}0=&\limsup_{l\to\infty}\limsup_{\mu\to\infty}\limsup_{n\to\infty} I_1^{n,\mu,l}\\+&\limsup_{l\to\infty}\limsup_{\mu\to\infty}\limsup_{n\to\infty} \iOT{w} \psi_l(u_n)   A(x,\nabla T_k( u_n))\cdot \nabla(T_k(u_n)-(T_k(u)\mathds{1}_{(0,T)})_\mu) \,dx\,dt+0.\end{split}\]
When we take into account~\eqref{limI1>0}, then the above line becomes\begin{equation*}
\limsup_{l\to\infty}\limsup_{\mu\to\infty}\limsup_{n\to \infty}\iOT {w}\psi_l(u_n) A(x,\nabla T_k(u_n))\cdot \nabla(T_k(u_n)-(T_k(u)\mathds{1}_{(0,T)})_\mu)\,dx\,dt\leq 0.
\end{equation*}
Note that due to (A2) we have $A(x,\nabla T_k(u_n))\cdot \nabla(T_k(u_n)))\geq 0$ and $A(x,0)=0$. Therefore, for sufficiently large $l,\mu,n$, since $w,\psi_l\geq 0$, and~\eqref{psil}, we have
\[\iOT w A(x,\nabla T_k(u_n))\cdot \nabla(T_k(u_n))\,dx\,dt\leq 
\iOT w A(x,\nabla T_k(u_n))\cdot \nabla(T_k(u)\mathds{1}_{(0,T)})_\mu\,dx\,dt.\]
On the right-hand side above we use~\eqref{TkmuinEm} and~\eqref{conv:n:ATkutMs}, and then for sufficiently large $\mu$
\[\limsup_{n\to\infty}\iOT w A(x,\nabla T_k(u_n))\cdot \nabla(T_k(u_n))\,dx\,dt\leq
\iOT w \Ak\cdot \nabla(T_k(u)\mathds{1}_{(0,T)})_\mu\,dx\,dt.\]
Recall that $\nabla (T_k (u))_\mu\xrightarrow{M}\nabla T_k (u)$. Therefore by Definition~\ref{def:convmod} ii), the sequence $\{M(x,\nabla(T_k (u)\mathds{1}_{(0,T)})_\mu/\lambda)\}_\mu$ is uniformly bounded in $L^1(\OT;\rn)$ for some $\lambda$ and consequently, by~Lemma~\ref{lem:unif}  $\{\nabla(T_k (u)\mathds{1}_{(0,T)})_\mu\}_\mu$ is uniformly integrable. Hence the Vitali Convergence Theorem (Theorem~\ref{theo:VitConv}) gives 
\[\lim_{\mu\to 0}\int_\OT w\Ak \cdot \nabla (T_k(u)\mathds{1}_{(0,T)})_\mu \,dx=\int_\OT w \Ak \cdot \nabla T_k(u) \,dx.\]
Consequently, we obtain~\eqref{limsup<}. Following the monotonicity argument, as in the proof of Proposition~\ref{prop:bound}, we prove that  \begin{equation*}
A(x,\nabla (T_k( u )))={\Ak} \qquad \text{a.e.}\quad\text{in}\quad \Omega_T.
\end{equation*}
 Monotonicity assumption (A3) of~$A$ implies
\[(A(x,\eta)-A(x,\nabla (T_k(u_n)))\cdot(\eta-\nabla(T_k( u_n)))\geq 0\qquad\text{a.e.  in }\OT,\]
for any $\eta\in L^\infty(\OT;\rn)\subset E_M(\OT;\rn)$. Since $A(x,\eta)\in L_{M^*}(\OT,\rn)= (E_{M}(\OT,\rn))^*$, we pass to the limit with $n\to\infty$ and take into account~\eqref{limsup<} to conclude that
\begin{equation}
\label{mono:int2}
\iOT w (A(x,\eta)-\Ak)\cdot(\eta-\nabla (T_k(u)))\,dx\,dt\geq 0.
\end{equation}

Let us define \begin{equation}
\label{omm2}
\Omega_T^{k,K}=\{(t,x)\in\OT:\ |\nabla T_k( u)|\leq K\quad\text{a.e. in }\OT\}.
\end{equation} 

We are going to show that
\begin{equation}
\label{lim:id:omj2}
A(x,\nabla (T_k(u)))=\Ak\qquad \text{a.e.}\quad\text{in}\quad \Omega_T^{k,j}
\end{equation}
for arbitrary $j$. This implies the equality a.e. in $\Omega_T$, i.e.~\eqref{Ak=ATk}.

We fix arbitrary $0<j<i$ and $z\in L^\infty(\OT;\rn)$.  Consider a parameter $h\in(0,1)$ and choose 
\[\eta=\nabla u\mathds{1}_{\Omega_T^{k,i}}+
hz\mathds{1}_{\Omega_T^{k,j}}.\]
Then in~\eqref{mono:int2} we have
\begin{equation*}
 - \int_{\OT\setminus\Omega_T^{k,i}}(A(x,0)-\Ak) \nabla u \,dx\,dt +h\int_{ \Omega_T^{k,j}} (A(x, \nabla (T_k(u))+h z )-\Ak)z  \, dx\,dt\geq  0.
\end{equation*} 
which becomes
\begin{equation}
\label{po-mon2} 
 \int_{\OT\setminus\Omega_T^{k,i}} \Ak \cdot \nabla u \,dx\,dt +h \int_{ \Omega_T^{k,j}} (A(x,\nabla  (T_k(u))+h z )-\Ak)z  \, dx\,dt\geq  0,
\end{equation} 
because (A2) implies $A(x,0)=0$. The Fenchel-Young inequality applied to $|\Ak \cdot \nabla (T_k( u))|$ above ensures (via the Lebesgue Dominated Convergence Theorem) that the first integral on the left-hand side vanishes when $i\to \infty$. Therefore, after passing with $i\to \infty$ in~\eqref{po-mon2}, we obtain \begin{equation*} 
\int_{ \Omega_T^{k,j}} (A(x,\nabla  (T_k(u))+h z )-\Ak)\cdot z \,  dx\, dt\geq  0\qquad\forall_{h \in(0,1)}
\end{equation*}
where $z $ and $j$ are fixed. We are going to pass to the limit with $h \to 0$. 

Note that
\[A(x,\nabla  (T_k(u))+h z )\xrightarrow[h\to 0]{}A(x,\nabla  (T_k(u)))\quad\text{a.e. in}\quad \Omega_T^{k,j} \] 
and $\{A(x,\nabla  (T_k(u))+h z )\}_{h }$ is pointwise bounded on $\Omega_T^{k,j}$ (cf.~\cite[(9)]{pgisazg1}).  Lemma~\ref{lem:M*<M}  yields
\[  M^*\left(x,A(x,\nabla  (T_k(u))+h z )\right) \leq \frac{2}{c_A}  M\left(x,\frac{2}{c_A}(\nabla  (T_k(u))+h z )\right).\]
Since $\{ M(x,\frac{2}{c_A}(\nabla  (T_k(u))+h z ))\}_{h }$ is uniformly bounded in~$L^1(\Omega_T^{k,j})$ (cf.~\eqref{LinfinEM} and~\eqref{omm2}), also $\{M^*\left(x,A(x,\nabla  (T_k(u))+h z )\right)\}_{h }$ is uniformly bounded in~$L^1(\Omega^{k,j}_T)$. Hence, Lemma~\ref{lem:unif} gives uniform integrability of $\{A(x,\nabla  (T_k(u))+h z )\}_{h }$ on $\Omega_T^{k,j}$. When we notice that $|\Omega_T^{k,j}|<\infty$, we can apply the Vitali Convergence Theorem (Theorem~\ref{theo:VitConv}) to~get
\[A(x,\nabla  (T_k(u))+h z )\xrightarrow[h\to 0]{}A(x,\nabla  (T_k(u)))\quad\text{in}\quad L^1(\Omega_T^{k,j};\rn).\] 
Thus
\begin{equation*} \int_{ \Omega_T^{k,j}} (A(x,\nabla  (T_k(u))+h z )-\Ak)\,z  w\, dx\, dt\xrightarrow[h\to 0]{} \int_{ \Omega_T^{k,j}} (A(x, (T_k(u)))-\Ak)\,z  w \, dx\,dt.
\end{equation*}
Consequently,
\begin{equation*}  \int_{ \Omega_T^{k,j}} (A(x,\nabla  (T_k(u)))-\Ak)\,z w\,  dx\, dt\geq 0,
\end{equation*}
for any $z \in L^\infty(\OT;\rn)$. Let us take 
\[z =\left\{\begin{array}{ll}-\frac{A(x,\nabla  (T_k(u)))-\Ak}{|A(x,\nabla  (T_k(u)))-\Ak|}&\ \text{if}\quad A(x,\nabla  (T_k(u)))-\Ak\neq 0,\\
0&\ \text{if}\quad A(x,\nabla  (T_k(u)))-\Ak= 0.
\end{array}\right.\]
We obtain 
\begin{equation*}  \int_{ \Omega_T^{k,j}}w |A(x,\nabla  (T_k(u)))-\Ak| dx\,dt\leq 0,
\end{equation*}
hence~\eqref{lim:id:omj2} holds. Taking into account that nonnegative $w\in C_c^\infty([0,T))$ is arbitrary, we get~\eqref{Ak=ATk}, which completes the proof.\end{proof}

\section{The proof of existence of renormalized solutions}\label{sec:mainproof} 
\begin{proof}[Proof of Theorem~\ref{theo:main}] Obviously, when $u_n$ solves~\eqref{eq:bound} its limit $u$ satisfies condition (R1), due to Proposition~\ref{prop:convTk} and Proposition~\ref{prop:convTkII}. The remaining (R2)-(R3) require more arguments.

\textbf{Condition (R3). Weak convergence of $A(x,\nabla T_k(u_n))\cdot \nabla T_k(u_n)$.} The aim now is to prove the key convergence for condition~(R3), namely\begin{equation}
\label{adtkn}A(x,\nabla T_k(u_n))\cdot \nabla T_k(u_n)\xrightharpoonup{} A(x,\nabla T_k(u))\cdot \nabla T_k(u) \quad \text{weakly in }L^1(\OT). 
\end{equation}
The reasoning involves the Chacon Biting Lemma and the Young measure approach. First we observe that  the sequence $\{[A(x,\nabla T_k(u_n))-A(x,\nabla T_k(u ))]\cdot [ \nabla T_k(u_n)- \nabla T_k(u_n)]\}_n$ is uniformly bounded in $L^1(\OT)$. Indeed,
\[\begin{split}&\iOT[A(x,\nabla T_k(u_n))-A(x,\nabla T_k(u ))]\cdot [ \nabla T_k(u_n)- \nabla T_k(u_n)]\,dx\,dt\leq\\ 
&\qquad\leq \iOT A(x,\nabla T_k(u_n)) \nabla T_{k}(u_n) dx\,dt
+\int_\OT A(x,\nabla T_k(u_n)) \nabla T_{k}(u )  dx\,dt+\\
&\qquad+\int_\OT A(x,\nabla T_{k}(u )) \nabla T_{k}(u_n) dx\,dt+\int_\OT A(x,\nabla T_{k}(u )) \nabla T_{k}(u ) dx\,dt=II_1+II_2+II_3+II_4, 
\end{split}\]
where $II_1$ is uniformly bounded due to~\eqref{apriori:thetai},  in the case of $II_2$ and $II_3$ the Fenchel-Young inequality and~\eqref{apriori:thetai} gives boundedness, while $II_4$ is independent of $n$.  

The monotonicity of $A(x,\cdot)$, uniform boundedness of $\{[A(x,\nabla T_k(u_n))-A(x,\nabla T_k(u ))]\cdot [ \nabla T_k(u_n)- \nabla T_k(u_n)]\}_n$ in $L^1(\OT)$, and Theorem~\ref{theo:Youngmeas} combined with Theorem~\ref{theo:bitinglemma1} give, up to~a~subsequence, convergence
\begin{equation}
\label{110}
\begin{split}0&\leq {w} [A(x,\nabla T_k(u_n) )-A(x,\nabla T_k(u ))]\cdot [\nabla T_{k}(u_n)-\nabla T_{k}(u )]\\
&\xrightarrow{b} w \int_\rnt [A(x,\lambda)-A(x,\nabla T_{k}(u ))]\cdot [\lambda-\nabla T_{k}(u )]d\nu_{t,x}(\lambda),\end{split}
\end{equation}
where $\nu_{t,x}$ denotes the Young measure generated by the sequence $\{\nabla T_{k}(u_n)\}_n$.

Since $\nabla T_{k}(u_n)\xrightharpoonup{}\nabla T_{k}(u )$ in $L^1(\OT)$, we have $\int_\rnt\lambda \, d\nu_{t,x}(\lambda)=\nabla T_{k}(u)$ for a.e. $t\in(0,T)$ and a.e. $x\in\Omega$. Then
\[\int_\rnt  A(x,\nabla T_k(u) ) \cdot [\lambda-\nabla T_{k}(u )]d\nu_{t,x}(\lambda)=0\]
and the limit in~\eqref{110} is equal for a.e. $t\in(0,T)$ and a.e. $x\in\Omega$ to
\begin{equation}\begin{split}
\label{111} {w} \int_\rnt [A(x,\lambda)-A(x,\nabla T_{k}(u ))]\cdot [\lambda-\nabla T_{k}(u )]d\nu_{t,x}(\lambda)=\\
{w}\int_\rnt  A(x,\lambda) \cdot  \lambda\, d\nu_{t,x}(\lambda)-{w}\int_\rnt  A(x,\lambda) \cdot \nabla  T_{k}(u )d\nu_{t,x}(\lambda).\end{split}
\end{equation}

Uniform boundedness of the sequence $\{ A(x,\nabla T_k(u_n))\cdot \nabla T_{k}(u_n) \}_n$ in $L^1(\OT)$  enables us  to apply once again Theorem~\ref{theo:Youngmeas} combined with Theorem~\ref{theo:bitinglemma1} to obtain
\begin{equation*}
 A(x,\nabla T_k(u_n))\cdot \nabla T_{k}(u_n)\xrightarrow{b} \int_\rnt  A(x,\lambda) \cdot  \lambda\,d\nu_{t,x}(\lambda).
\end{equation*}
Moreover, assumption (A2) implies $A(x,\nabla T_k(u_n))\cdot \nabla T_{k}(u_n)\geq 0$. Therefore, due to~\eqref{111} and~\eqref{110}, we have
\[\limsup_{n\to\infty} A(x,\nabla T_{k}(u_n) ) \nabla T_{k}(u_n)\geq \int_\rnt  A(x,\lambda) \cdot  \lambda\,d\nu_{t,x}(\lambda).\]
Taking into account that in~\eqref{conv:n:AT:id} we can put ${\cal A}_k=A(x,\nabla T_{k}(u) )=\int_\rnt  A(x,\lambda) \,d\nu_{t,x}(\lambda)$, the above expression implies\[\nabla T_{k}(u) \int_\rnt  A(x,\lambda) \,d\nu_{t,x}(\lambda)\geq \int_\rnt  A(x,\lambda) \cdot  \lambda\,d\nu_{t,x}(\lambda).\]

When we apply it, together with~\eqref{111}, the limit in~\eqref{110} is non-positive. Hence,
\begin{equation*} [A(x,\nabla T_{k}(u_n) )-A(x,\nabla T_{k}(u ))]\cdot [\nabla T_{k}(u_n)-\nabla T_{k}(u )]\xrightarrow{b} 0.
\end{equation*}
Observe further that $A(x,\nabla T_{k}(u ))\in L_{M^*}(\Omega;\rn)$ and we can choose ascending family of sets $E^{k}_j$, such that $| E^{k}_j|\to 0$ for $j\to \infty$ and $A(x,\nabla T_{k}(u ))\in L^\infty(\Omega\setminus E^{k}_j).$ Then, since $\nabla T_{k}(u_n)\xrightharpoonup{}\nabla T_{k}(u )$, we get\begin{equation*} A(x,\nabla T_{k}(u )) \cdot [\nabla T_{k}(u_n)-\nabla T_{k}(u )]\xrightarrow{b}  0
\end{equation*}
and similarly we conclude
\begin{equation*} A(x,\nabla T_{k}(u_n ))\cdot\nabla T_{k}(u )\xrightarrow{b}   A(x,\nabla T_{k}(u))\cdot\nabla T_{k}(u ).
\end{equation*}
Summing it up we get
\begin{equation*}
A(x,\nabla T_{k}(u_n ))\cdot\nabla T_{k}(u_n )\xrightarrow{b}  A(x,\nabla T_{k}(u))\cdot\nabla T_{k}(u ).
\end{equation*}
Let us point out that (A2) ensures that both --- the right and the left--hand sides are nonnegative. Recall that Theorem~\ref{theo:bitinglemma} together with~\eqref{limsup<} and~\eqref{conv:n:ATkutMs} results in~\eqref{adtkn}.

\medskip

\textbf{Condition (R3). Conclusion. }   Note that $\nabla u_n=0$ a.e. in $\{|u_n|\in\{l,l+1\}\}$. Then~\eqref{eq:contr:rad:n} implies
\begin{equation*}
 \lim_{l\to\infty}\sup_{n>0}\int_{\{l-1<|u_n|<l+2\}}A(x,\nabla u_n)\cdot\nabla u_n\,dx=0.
\end{equation*}
For $g^l:\r\to\r$ defined by
\[g^l(s)=\left\{\begin{array}{ll}1&\text{if }\ l\leq| s|\leq l+1,\\
0&\text{if }\ |s|<l-1\text{ or } |s|> l+2,\\
\text{is affine} &\text{otherwise},
\end{array}\right.\]
we have
\begin{equation}
\label{128}
\int_{\{l<|u|<l+1\}}A(x,\nabla u)\cdot\nabla u\,dx\,dt\leq \int_{\OT}g^l(u)A(x,\nabla T_{l+2}( u))\cdot\nabla  T_{l+2}( u)\,dx\,dt.
\end{equation}
Let us remind that we know that $u_n\to u$ a.e. in $\OT$ (cf.~\eqref{conv:usae}) and $\lim_{l\to\infty}|\{x:|u_n|>l\}|=0$ (cf.~\eqref{conv:umeas}). Moreover, we have weak convergence~\eqref{adtkn}, $A(x,\nabla T_{l+2}( u_n))\cdot\nabla  T_{l+2}( u_n)\geq 0$ and function $g^l$ is continuous and bounded. Thus, we infer that we can estimate the limit of the right-hand side of~\eqref{128} in the following way
\[\begin{split}
0&\leq \lim_{l\to\infty} \int_{\{l-1<|u|<l+2\}}A(x,\nabla u)\cdot\nabla u\,dx\,dt\leq \lim_{l\to\infty}\int_{\Omega}g^l(u)A(x,\nabla T_{l+2}( u))\cdot\nabla  T_{l+2}( u)\,dx\,dt=\\
&= \lim_{l\to\infty} \lim_{n\to\infty} \int_{\Omega}g^l(u_n)A(x,\nabla T_{l+2}(u_n))\cdot\nabla T_{l+2}(u_n)\,dx\,dt\leq\\
&\leq \lim_{l\to\infty}  \lim_{n\to\infty}\int_{\{l-1<|u_n|<l+2\}}A(x,\nabla T_{l+2}(u_n))\cdot\nabla T_{l+2}(u_n)\,dx\,dt=0,
\end{split}\]
where the last equality comes from~\eqref{eq:contr:rad:n}. Hence, our solution $u$ satisfies condition (R3).

\bigskip

\textbf{Condition (R2). }  For the proof of (R2) we need to apply Lemma~\ref{lem:intbyparts} for~\eqref{eq:bound}, arbitrary $h\in C_c^1(\R)$ and $\xi\in{C_c^\infty}([0,T)\times\Omega)$. Then
\begin{equation}
\label{fin-weak-n}
-\int_{\OT} \left(\int_{u_{0,n}}^{u_n} h(\s)d\s\right) \partial_t \xi \ \,dx\,dt+\int_{\OT}A(x,\nabla u_n)\cdot \nabla (h(u_n)\xi) \,dx\,dt=\int_{\OT}T_n(f) h(u_n)\xi \,dx\,dt.
\end{equation}
To pass to the limit with $n\to\infty$  side above we fix $R>0$ such that $\supp\, h\subset [-R,R]$. The right-hand converges to the desired limit due to the Lebesgue Dominated Convergence Theorem since $T_n f\to f$ in $L^1(\OT)$ and $\{h(u_n)\}_n$ is uniformly bounded.

To pass to the limit on the left-hand side  we notice that  we have there
\[\lim_{n\to\infty}-\int_{\OT} \left(\int_{u_{0,n}}^{u_n} h(\s)d\s\right) \partial_t \xi \ \,dx\,dt=-\int_{\OT} \left(\int_{u_{0 }}^{u } h(\s)d\s\right) \partial_t \xi \ \,dx\,dt, \]
where the equality is justified by the continuity of the integral. As for the second expression, we can write 
\[\begin{split}&\int_{\OT}A(x,\nabla u_n)\cdot \nabla (h(u_n)\xi) \,dx\,dt 
= \int_{\OT}h'(T_R(u_n)) A(x,\nabla T_R(u_n))  \nabla T_R(u_n)\,\xi \,dx\,dt\\&\qquad+\int_{\OT}h(T_R(u_n))A(x,\nabla T_R(u_n))\cdot \nabla  \xi  \,dx\,dt=III^n_1+III^n_2.\end{split}\]
 Recall weak convergence of $A(x,\nabla T_k(u_n))\cdot \nabla T_k(u_n)$ in $L^1(\OT)$~\eqref{adtkn}. Since $h'(u_n)\xi\to h'(u)\xi$ a.e. in~$\OT$ and \[\|h'(u_n)\xi\|_{L^\infty(\OT)}\leq \|h'(u_n) \|_{L^\infty(\OT)}\| \xi\|_{L^\infty(\OT)},\]
we pass to the limit with $n\to\infty$ in $III^n_1$. Whereas to complete the case of $III^n_2$ we observe that Proposition~\ref{prop:convTkII}  implies weak convergence of $A(x,\nabla T_R(u_n))$ in $L^1(\OT)$ as $n\to\infty$. Moreover, $\{h(T_R(u_n))\}_n$ converges a.e. in~$\OT$ to $h(T_R(u))$ and is uniformly bounded in $L^\infty(\OT)$, so we can pass to the limit. Altogether  we have
\[\lim_{n\to\infty}(III^n_1+III^n_2)=\int_{\OT}h'(T_R(u )) A(x,\nabla u )  \nabla T_R(u )\,\xi \,dx\,dt +\int_{\OT}h(T_R(u ))A(x,\nabla T_R(u ))\cdot \nabla  \xi  \,dx\,dt.\]
Therefore, all the expressions of~\eqref{fin-weak-n} converge to the limits as expected in (R2).

We already proved that $u$ satisfies (R1), (R2), and (R3), hence it is a~renormalized solution. Uniqueness is a direct consequence of the comparison principle (Proposition~\ref{prop:comp-princ}).\end{proof}

\appendix
\addcontentsline{toc}{section}{Appendices}
\section*{Appendices}

\section{Basics}\label{ssec:Basics}

\begin{defi}[$N$-function]\label{def:Nf} Suppose $\Omega\subset\rn$ is an open bounded set. A~function   $M:\Omega\times\rn\to\r$ is called an $N$-function if it satisfies the
following conditions:
\begin{enumerate}
\item $ M$ is a Carath\'eodory function (i.e. measurable with respect to $x$ and continuous with respect to the last variable), such that $M(x,0) = 0$, $\inf_{x\in\overline{\Omega}}M(x,\xi)>0$ for $\xi\neq 0$, and $M(x,\xi) = M(x, -\xi)$ a.e. in $\Omega$,
\item $M(x,\xi)$ is a convex function with respect to $\xi$,
\item $\lim_{|\xi|\to 0}\mathrm{ess\,sup}_{x\in\overline{\Omega}}\frac{M(x,\xi)}{|\xi|}=0$,
\item $\lim_{|\xi|\to \infty}\mathrm{ess\,inf}_{x\in\overline{\Omega}}\frac{M(x,\xi)}{|\xi|}=\infty$.
\end{enumerate}
\end{defi}

\begin{defi}[Complementary function] \label{def:conj} 
The complementary~function $M^*$ to a function  $M:\Omega\times\rn\to\r$ is defined by
\[M^*(x,\eta)=\sup_{\xi\in\rn}(\xi\cdot\eta-M(x,\xi)),\qquad \eta\in\rn,\ x\in\Omega.\]
\end{defi}

\begin{lem} If $M$ is an $N$-function and $M^*$ its complementary, we have 
 the Fenchel-Young inequality \begin{equation}
\label{inq:F-Y}|\xi\cdot\eta|\leq M(x,\xi)+M^*(x,\eta)\qquad \mathrm{for\ all\ }\xi,\eta\in\rn\mathrm{\ and\ }x\in\Omega.
\end{equation}
\end{lem}
\begin{lem}\label{lem:M*<M} Suppose $M$ and $A$ are such that (A2) is satisfied, then 
\begin{equation*}
  M^*(x,A(x,\eta)) \leq \frac{2}{c_A}  M\left(x, \frac{2}{c_A}\eta \right) \quad\text{for}\quad \eta\in L^\infty(\Omega;\rn).\end{equation*} 
\end{lem}
\begin{proof}
Since $M^*$ is convex, $M^*(x,0)=0$ and $c_A\in (0,1]$, we notice that
\[M^*\left(x,\frac{c_A}{2}A\left(x,\eta\right)\right)\leq \frac{c_A}{2}M^*(x,A(x,\eta)).\]
Taking this into account together with~(A2) and~\eqref{inq:F-Y} we have
\[\begin{split}
c_A\left(M(x,\eta)+M^*(x,A(x,\eta))\right)\leq \frac{c_A}{2} A(x,\eta)\cdot \frac{2}{c_A}\eta &\leq M\left(x,\frac{2}{c_A}\eta\right)+M^*\left(x,\frac{c_A}{2}A(x,\eta)\right)\leq\\ &\leq M\left(x,\frac{2}{c_A}\eta\right)+\frac{c_A}{2}M^*\left(x,A(x,\eta)\right).\end{split}\]
We ignore $M(x,\eta)>0$ on the left-hand side above and rearrange the rest terms   to get the claim.\end{proof}

\begin{rem}\label{rem:2ndconj} For any function $f:\r^M\to\r$ the second conjugate function $f^{**}$ is convex and $f^{**}(x)\leq f(x)$. In fact,  $f^{**}$ is a convex envelope of $f$, namely it is the biggest convex function smaller or equal to~$f$.
\end{rem}

\begin{lem}[cf.~\cite{martin,pgisazg1}]\label{lem:Mass} Suppose  a cube $\Qd$ is an arbitrary one defined in (M) with $\delta_0=1/(8\sqrt{N})$ and function $M:\rn\times[0,\infty)\rightarrow[0,\infty)$  is log-H\"older continuous, that is there exist constants $a_1>0$ and $b_1\geq 1$, such that for all $x,y\in\Omega$ with $|x-y|\leq \frac{1}{2}$ and all $\xi\in\rn$ we have~\eqref{M2'}.  Let us consider function  $ \Mjd $ given by~\eqref{Mjd} and its greatest convex minorant $(\Mjd)^{**}$. Then there exist constants $a,c>0$, such that~\eqref{M2} is satisfied. 
\end{lem}

\begin{defi}[$\Delta_2$-condition]\label{def:D2}
 We say that an $N$-function $M:\Omega\times\rn\to\r$ satisfies $\Delta_2$ condition if for a.e. $x\in\Omega$, there exists a constant $c>0$ and nonnegative integrable function $h:\Omega\to\r$ such that
\begin{equation*}
 M(x,2\xi)\leq cM(x,\xi)+h(x).
\end{equation*}
\end{defi}

\begin{lem}\label{lem:constr} For any $N$-function $M:\Omega\times\rn\to\r$ there exists a function $\dm:\r_+\cup\{0\}\to\r$, such that
\begin{equation}
\label{dm:leq:M}
\dm(s)\leq \inf_{x\in\Omega,\ |\xi|=s} M(x,\xi)\qquad\forall_{s\in\r_+\cup\{0\}}.
\end{equation}
\end{lem}
\begin{proof}
We consider the $N$-function $\dm:\r_+\cup\{0\}\to\r$ defined as follows. Let
\begin{equation*}
m_*(r)=\left(\inf_{x\in\Omega,\ |\xi|=r} M(x,\xi)\right)^{**}.
\end{equation*} 
Then, let $\dm$ be a solution to the differential equation
\[\dm'(s)=\left\{\begin{array}{ll}
m'_*(s)&\text{for }s:\ m'_*(s)\leq \alpha\frac{m_*(s)}{s},\\
\alpha\frac{\dm(s)}{s}&\text{for }s:\ m'_*(s)> \alpha\frac{m_*(s)}{s},
\end{array}\right.\]
with the initial condition $\dm(0)=0=m_*(0)$ and a certain $\alpha>1$. Note that $\dm'(s)\leq m'_*(s)$  for every $s$, so $\dm(s)\leq m_*(s)$. Due to Lemma~\ref{rem:2ndconj} also $ m_*(s)\leq  \inf_{x\in\Omega,\ |\xi|=s} M(x,\xi)$ for every $s$. Thus we have~\eqref{dm:leq:M}. Moreover, by~\cite[Chapter~II.2.3, Theorem~3, point~1. (ii)]{rao-ren} $\dm$ satisfies $\Delta_2$-condition. 
\end{proof}

The vital tool in our study is the following modular Poincar\'{e}-type inequality being a direct consequence of \cite[Theorem~2.3]{pgisazg1}.
\begin{theo}[Modular Poincar\'e inequality,~\cite{pgisazg1}]\label{theo:Poincare} 
Let $P:\rp\to\rp$ be an arbitrary function satisfying $\Delta_2$-condition and $\Omega\subset\rn$ be a bounded domain,
then there exist $c_P=c(\Omega,N,P)>0$ such that for every $g\in W^{1,1}(\OT)$, such that $\int_\OT P(|\nabla g|)\,dx\,dt<\infty$, we have
\[\int_\OT P(|g|)\,dx\,dt\leq c_P
\int_\OT P(|\nabla g|)\,dx\,dt.\]
\end{theo}

\begin{defi}[Uniform integrability] We call a sequence  $\{f_n\}_{n=1}^\infty$ of measurable functions $f_n:\Omega\to \rn$ 
 uniformly integrable if
\[\lim_{R\to\infty}\left(\sup_{n\in\mathbb{N}}\int_{\{x:|f_n(x)|\geq R\}}|f_n(x)|dx\right)=0.\]  
 \end{defi}
 
\begin{defi}[Modular convergence]\label{def:convmod}
We say that a sequence $\{\xi_i\}_{i=1}^\infty$ converges modularly to $\xi$ in~$L_M(\Omega;\rn)$ (and denote it by $\xi_i\xrightarrow[i\to\infty]{M}\xi$), if 
\begin{itemize}
\item[i)] there exists $\lambda>0$ such that
\begin{equation*}
\int_{\Omega}M\left(x,\frac{\xi_i-\xi}{\lambda}\right)dx\to 0,
\end{equation*}
equivalently
\item[ii)] there exists $\lambda>0$ such that 
\begin{equation*}
 \left\{M\left(x,\frac{\xi_i}{\lambda}\right)\right\}_i \ \text{is uniformly integrable in } L^1(\Omega)\quad \text{and}\quad \xi_i\xrightarrow[]{i\to\infty}\xi \ \text{in measure}.
\end{equation*}
\end{itemize}
\end{defi}

\begin{lem}[Modular-uniform integrability,~\cite{gwiazda2}]\label{lem:unif}
Let $M$ be an $N$-function and $\{f_n\}_{n=1}^\infty$ be a~sequence of measurable functions such that $f_n:\Omega\to \rn$ and $\sup_{n\in\N}\int_\Omega M(x,f_n(x))dx<\infty$. Then the sequence $\{f_n\}_{n=1}^\infty$ is uniformly integrable.
\end{lem}

The following result can be obtained by the method of the proof of~\cite[Theorem~7.6]{Musielak}.
\begin{lem}[Density of simple functions, \cite{Musielak}]\label{lem:dens}
Suppose~\eqref{ass:M:int}. Then the set of simple functions integrable on $\OT$ is dense in $L_M(\OT)$ with respect to the modular topology.
\end{lem}

\begin{defi}[Biting convergence]\label{def:convbiting}
Let $f_n,f\in  L^1(\Omega)$ for every $n\in\N$. We say that a sequence $\{f_n\}_{n=1}^\infty$ converges in the sense of biting to $f$ in~$L^1(\Omega)$ (and denote it by $f_n\xrightarrow[]{b}f$), if  there exists a sequence of measurable $E_k$ -- subsets of $\Omega$, such that $\lim_{k\to\infty} |E_k|=0$, such that for every $k$ we have $f_n\to f$ in $L^1(\Omega\setminus E_k)$.
\end{defi}

To present basic information on the Young measures, let us denote the space of signed Radon measures with finite mass by ${\cal M}(\rn)$.
\begin{theo}[Fundamental theorem on the Young measures]\label{theo:Youngmeas}
Let $U\subset\rn$ and $z_j:U\to\rn$ be a~sequence of measurable functions. Then there exists a~subsequence $\{z_{j,k}\}$ and a~family of weakly-* measurable maps $\nu_x:U\to{\cal M}(\rn)$, such that:
\begin{itemize}
\item $\nu_x\geq 0$, $\|\nu_x\|_{{\cal M}(\rn)}=\int_\rn d\nu_x\leq 1$ for a.e. $x\in U$.
\item For every $f\in C_0(\rn)$, we have $f(z_{j,k})\xrightharpoonup[]{*} \bar{f}$ in $L^\infty(U)$. Moreover,  $\bar{f}(x)=\int_\rn f(\lambda) \, d\nu_x(\lambda)$.
\item Let $K\subset\rn$ be compact. Then $\supp\,\nu_x\subset K,$ if $dist(z_{j,k},K)\to 0$ in measure.
\item $\|\nu_x\|_{{\cal M}(\rn)}=1$ for a.e. $x\in U$ if and only if the tightness condition is satisfied, that is $\lim_{R\to\infty}\sup_k|\{|z_{j,k}|\geq R\}|=0$.
\item If the tightness condition is satisfied, $A\subset U$ is measurable, $f\in C(\rn)$, and $\{f(z_{j,k})\}$ is relatively weakly compact in $L^1 (A)$, then $f(z_{j,k})\xrightharpoonup[]{} \bar{f}$ in $L^1(A)$ and  $\bar{f}(x)=\int_\rn f(\lambda) \, d\nu_x(\lambda)$.
\end{itemize}  
The family of maps $\nu_x:U\to {\cal M}(\rn)$ is called the
Young measure generated by the sequence $\{z_{j,k}\}$.
\end{theo}

\begin{theo}[The Chacon Biting Lemma, cf. Theorem~6.6 in \cite{pedr}]\label{theo:bitinglemma1}Let the sequence $\{f_n\}_n$ be uniformly bounded in $L^1(\Omega)$. Then there exists $f\in L^1(\Omega)$, such that $f_n\xrightarrow[]{b}f$.
\end{theo}

The consequence of the above result is the following, cf.~\cite[Lemma~6.9]{pedr}.

\begin{theo}\label{theo:bitinglemma}Let $f_n\in   L^1(\Omega)$ for every $n\in\N$,   $f_n(x)\geq 0$ for every $n\in\N$ and a.e. $x$ in $\Omega$. Moreover, suppose $f_n\xrightarrow[]{b}f$ and $\limsup_{n\to\infty}\int_\Omega f_n dx\leq \int_\Omega f dx.$ Then  $f_n\xrightharpoonup{}f$ in $L^1(\Omega)$ for $n\to\infty$.
\end{theo}
\begin{lem}[The Young Inequality for convolutions]\label{lem:NYS}
Suppose $q,r,s\geq 1$, $1/q+1/r+1/s=2$, and $u\in L^q,$ $v\in L^r$, $\psi\in L^s$. Then
\[\left|\int_\rn \psi(x)\cdot(u*v)(x)\,dx\right|\leq \|u\|_{L^q}\|v\|_{L^r}\|\psi\|_{L^s}.\]
\end{lem}

\begin{theo}[The Vitali Convergence Theorem]\label{theo:VitConv} Let $(X,\mu)$ be a positive measure space, $\mu(X)<\infty $, and $1\leq p<\infty$. If $\{f_{n}\}$ is uniformly integrable in $L^p_\mu$,   $f_{n}(x)\to f(x)$ in measure  and $|f(x)|<\infty $  a.e. in $X$, then  $f\in  {L}^p_\mu(X)$
and  $f_{n}(x)\to f(x)$ in  ${L}^p_\mu(X)$.
\end{theo} 




\section{Approximation}\label{ssec:Appr} 

Let $\kappa_\delta=1- {\delta}/{R}$. For a measurable function $\xi:[0,T]\times\rn\to\rn$ with $\mathrm{supp}\,\xi\subset[0,T]\times\Omega$, we define 
\begin{equation}
\label{Sdxi}S_\delta(\xi(t,x)) =
 \int_\Omega \rho_\delta( x-y)\xi (t,\kappa_\delta y)dy,
\end{equation} 
where $ \rho _\delta(x)=\rho (x/\delta)/\delta^N$ is a standard regularizing kernel on $\rn$  (i.e. $\rho \in C^\infty(\rn)$,
$\mathrm{supp}\,\rho \subset\subset B(0, 1)$ and $\iO \rho (x)dx = 1$, $\rho (x) = \rho (-x)$). Let us notice that $\xi_\delta\in C_c^\infty(\rn;\rn)$.

\begin{lem}\label{lem:step2prev} Suppose $M$ is an $N$-function satisfying condition~(M) and $\Omega$ is a star-shape domain with respect to a ball $B(0,R)$ for some $R>0$. Let $S_\delta$ be given by~\eqref{Sdxi} and $\delta<\delta_0$. Then there exist a constant $C>0$ independent of $\delta$ such that
\begin{equation} 
\label{in:Md<M}
\int_\OT M(x,S_\delta \xi(t,x))\, dx\,dt\leq C
\int_\OT M(x, \xi(t,x))\, dx\,dt
\end{equation}
for every $\xi\in L_M(\OT;\rn)\cap L^1(\OT)$.
\end{lem}
\begin{proof}  For $0 < \delta < R$ it holds that
\[\overline{\left(1-\frac{\delta}{R}\right) \Omega + \delta B(0, 1)} \subset \Omega.\]
Therefore, $S_\delta\xi\in L^\infty (0,T; C_c^\infty(\Omega))$. Let $0<\delta<\delta_1:=\min\{R/{4},\delta_0\}$ and $\{\Qd\}_{j=1}^{N_\delta}$ be a family defined in~(M). We consider $M_j^\delta(\xi)$ given by~\eqref{Mjd} and $\Mss$, see~Remark~\ref{rem:2ndconj}. Since $M(x,\xi_\delta(x))=0$ whenever $\xi_\delta(x)=0$, we have
\begin{equation}
\label{M:div-mult}\begin{split}
\int_0^T\iO M(x,S_\delta(\xi(t,x)))dxdt=\sum_{j=1}^{N_\delta}\int_0^T \iQd M(x,S_\delta(\xi(t,x)))dxdt=\\=\sum_{j=1}^{N_\delta}\int_0^T \iQdn \frac{M(x,S_\delta(\xi(t,x)))}{\Msdx}{\Msdx}dxdt.\end{split}
\end{equation}

Our aim is to show now the following uniform bound
\begin{equation}
\label{M/M<c}\frac{M(x,S_\delta(\xi(t,x)))}{\Msdx}\leq C
\end{equation}
for  sufficiently small $\delta>0$, $x\in\Qd\cap\Omega$ with $c$ independent of $\delta,x$ and $j$. 

 Let us fix an arbitrary cube and take $x\in \Qd$. For sufficiently small $\delta$ (i.e. $\delta< \delta_0$), due to~\eqref{M2}, we obtain \begin{equation}
\label{M/M<xi}\frac{M(x,S_\delta(\xi(t,x)))}{\Msdx} 
\leq   c \left(1+ |S_\delta(\xi(t,x))|^{-\frac{a}{\log(3\delta\sqrt{N})}} \right).
\end{equation}

To estimate the right--hand side of~\eqref{M/M<xi} we consider $S_\delta$ given by~\eqref{Sdxi}. Denote \[K=\sup_{B(0,1)}|\rho (x)|.\]
Note that for any $x,y\in\Omega$   and each $\delta>0$  we have
\[\rho _\delta(x-y)\leq  {K}/{\delta^N}.\]
Without loss of generality it can be assumed that $\|\xi\|_{L^\infty(0,T; L^1(\Omega))}\leq 1$, so
\begin{equation}
\label{xidest}\begin{split}|S_\delta \xi(t,x)|& 
\leq \frac{K}{\delta^N }  \int_{\Omega }  |\xi (t,\kappa_\delta y)|dy \leq \frac{K }{\delta^N\kappa_\delta}\|\xi\|_{L^\infty(0,T; L^1(\Omega))}\leq \frac{2K}{\delta^N }.\end{split}
\end{equation}
Note that \[\left|   {\delta^N} \right|^{ \frac{a}{\log (b\delta )}}=\exp \frac{aN \log  \delta}{\log (b\delta )}\]
is bounded for $\delta\in [0,\delta_0]$. We combine this with~\eqref{M/M<xi} and~\eqref{xidest} to get\begin{equation*}
\frac{M(x,S_\delta(\xi(t,x)))}{ \Msdx } \leq c \left(1+ \left|\frac{2K}{\delta^N }\right|^{-\frac{a}{\log(b\delta )}} \right)\leq C.
\end{equation*}
Thus, we have obtained~\eqref{M/M<c}.

Now, starting from~\eqref{M:div-mult}, noting~\eqref{M/M<c}  and the fact that $\Mss$=0 if and only if $\xi=0$, we observe \[
\begin{split}
\int_0^T\iO M(x,S_\delta\xi(t,x))dx &=\sum_{j=1}^{N_\delta}\int_0^T \iQdn \frac{M(x,S_\delta\xi(t,x))}{\Msdx}{\Msdx}dxdt\leq \\
&\leq C\sum_{j=1}^{N_\delta}\int_0^T \iQdn {\Msdx}dxdt\leq\\
&\leq C\sum_{j=1}^{N_\delta}\int_0^T \iQd \   {{\Msd}\left( \int_{B(0,\delta)} \rho _\delta(y)\xi (t,\kappa_\delta (x-y))dy\right)}\mathds{1}_{\Qd\cap\Omega}(x) dxdt\leq\\
&\leq C\sum_{j=1}^{N_\delta}\int_0^T \int_\rn  \  {{\Msd}\left( \int_{B(0,\delta)} \rho _\delta(y)\xi (t,\kappa_\delta (x-y))\mathds{1}_{\Qd\cap\Omega}(x)dy\right)} dxdt\leq\\
&\leq C\sum_{j=1}^{N_\delta}\int_0^T \int_\rn  {{\Msd}\left( \int_{\rn} \rho _\delta(y)\xi (t,\kappa_\delta (x-y))\mathds{1}_{\tQd\cap\Omega}(x-y)dy\right)} dxdt.\end{split} 
\]
Note  that by applying the Jensen inequality  the right-hand side above can be estimated by the following quantity
\[ \begin{split} 
& \quad \ C\sum_{j=1}^{N_\delta} \int_0^T\int_\rn \int_{\rn} \rho _\delta(y) {{\Msd}\left( \xi (t,\kappa_\delta (x-y))\mathds{1}_{\tQd\cap\Omega}(x-y) \right)} dy\,dxdt\leq\\
& \leq C \| \rho _\delta\|_{L^1({B(0,\delta);\rn})}\sum_{j=1}^{N_\delta}\int_0^T\int_\rn {{\Msd}\left( \xi (t,\kappa_\delta z)\mathds{1}_{\tQd\cap\Omega}(z) \right)}  dzdt\leq\\
&\leq C  \sum_{j=1}^{N_\delta} \int_0^T \int_{\tQd\cap\Omega} {{\Msd}\left( \xi (t,\kappa_\delta z) \right)} dzdt.\end{split}
\] 
We applied inequality for convolution, boundedness of $\rho _\delta$, once again  the fact that $\Mss$=0 if~and only if~$\xi=0$. Then, by the definition of  $M_j^\delta(\xi)$~\eqref{Mjd} and properties of $\Mss$, see~Remark~\ref{rem:2ndconj}, we realize that
\[ \begin{split} C \sum_{j=1}^{N_\delta}\int_0^T  \int_{\tQd\cap\Omega} {{\Msd}\left( \xi (t,\kappa_\delta z) \right)} dzdt&\leq C\sum_{j=1}^{N_\delta} \int_0^T \int_{{\kappa_\delta}\Qd} {M\left(x, \xi (t,x) \right)}  dxdt  \leq\\
&\leq C\sum_{j=1}^{N_\delta} \int_0^T \int_{{2}\Qd} {M\left(x, \xi (t,x) \right)}  dxdt\leq  C(N)\int_0^T\int_\Omega {M\left(x, \xi (t,x) \right)} dxdt.\end{split}
\]
The last inequality above stands for computation of~a~sum taking into account the measure of~repeating parts of cubes.

We get~\eqref{in:Md<M} by summing up the above estimates.\end{proof}

\begin{proof}[Proof of Theorem~\ref{theo:approx}] If $\Omega$ is a bounded Lipschitz domain in~$\rn$, then there exists a finite family of open sets
$\{\Omega_i\}_{i\in I}$ and a finite family of balls $\{ B^i\}_{i\in I}$ such that 
$$\Omega=\bigcup\limits_{i\in I}\Omega_i$$
and every set $\Omega_i$ is star-shaped with respect to ball $B^i$ of radius $R_i$ (see e.g.~\cite{Novotny}). 
Let us  introduce the partition of unity $\theta_i$ with
$0\le\theta_i\le 1,$ $\theta_i\in C^\infty_0(\Omega_i), $ ${\rm supp} \,\theta_i=\Omega_i,$ $\sum_{i\in I}\theta_i( x)=1$
 for $x\in\Omega$. We define $Q_i:=(0,T)\times \Omega_i$.

\medskip
 
 Fix arbitrary $\vp \in \VTMi$. We will show that there exists a constant $\lambda>0$ such that for each $\ve$, there exist $\vp_\ve\in L^\infty(0,T; C_c^\infty(\Omega))$, such that  \[\iOT M\left(x,\frac{\nabla \vp_\ve-\nabla \vp}{\lambda}\right)dx\,dt<\ve.\] 
We contruct $\vp_\ve$ by analysis of $S_\delta(T_l\vp)$, where $S_\delta$ is defined in~\eqref{Sdxi} and truncation $T$ is given by~\eqref{Tk}. Namely,  we are going to show that there exists a constant $\lambda>0$ such that
\[\lim_{l\to\infty}\lim_{\delta\to 0^+}\iOT M\left(x, \frac{\nabla S_\delta(  T_l(\vp))- \nabla \vp}{\lambda}\right) dx dt = 0.\]

Since\[ \iOT M\left(x, \frac{\nabla S_\delta(  T_l(\vp))- \nabla \vp}{\lambda}\right) dx dt \leq\sum_{i\in I} \int_{Q_i} M\left(x, \frac{\nabla S_\delta(  T_l(\vp))- \nabla \vp}{\lambda}\right) dx dt \]
it suffices to prove it to prove convergence to zero of each integral from the right-hand side.

Let us consider a family of measurable sets  $\{ E_n \}_n$  such that $\bigcup_n E_n = \OT$ and a simple vector valued function 
\[E^n(t,x)=\sum_{j=0}^n \mathds{1}_{E_j}(t,x) \va_{j}(t,x),\]
converging modularly to $\nabla(T_l \vp )$ with $\lambda_3$ (cf.~Definition~\ref{def:convmod}) which exists due to Lemma~\ref{lem:dens}.  

Note that 
\[  \nabla S_\delta(  T_l \vp)- \nabla \vp=\left(\sum_{i\in I}\nabla S_\delta(\theta_i  T_l \vp)-S_\delta E^n\right) +(S_\delta E^n -E^n)  
	+ (E^n - \nabla(T_l \vp))+ (\nabla(T_l \vp)- \nabla  \vp) .\]

Convexity of $M(x, \cdot)$ implies 
	\begin{equation*}
	\begin{split}
	&\int_{Q_i} M \left( x, \frac{ \nabla S_\delta(  T_l \vp)- \nabla \vp }{ \lambda }\right) \,dx dt=\\
	& \leq	\frac{\lambda_1}{\lambda} \int_{Q_i}  M\left( x, 
	\frac{\nabla S_\delta(  T_l \vp) -S_\delta E^n}{ \lambda_1} \right) \,dx dt	+ \frac{\lambda_2}{\lambda} \int_{Q_i}  M\left( x, 
	\frac{S_\delta E^n -E^n }{\lambda_2 } \right) \,dx dt\\
	& + \frac{\lambda_3}{\lambda} \int_{Q_i}  M\left( x,  \frac{ E^n - \nabla(T_l \vp)}{\lambda_3} \right) \,dx dt+ \frac{\lambda_4}{\lambda} \int_{Q_i}  M\left( x,  \frac{ \nabla(T_l \vp)- \nabla  \vp }{\lambda_4} \right) \,dx dt=\\
	&=L^{l,n,\delta}_1+L^{l,n,\delta}_2+L^{l,n,\delta}_3+L^{l,n,\delta}_4,
	\end{split}
	\end{equation*}
where $\lambda= \sum_{i=1}^4\lambda_i$, $\lambda_i>0$.  We have $\lambda_3$ fixed already. Let us take $\lambda_1=\lambda_3$.

\medskip

We note that $T_l \vp \in \VTMi$  and for each $i\in I$ we have
 \[\theta_i\cdot T_l\vp\in L^\infty (Q_i)\cap L^\infty(0,T;L^2(\Omega_i))\cap L^1(0,T;W^{1,1}_0(\Omega_i)) \]
and 
\[\nabla (\theta_i T_l\vp)\ =\ T_l\vp \nabla \theta_i + \theta_i \nabla T_l\vp \ \in\ L_M(\Omega;\rn).\]
Furthermore, $\sum_{i\in I} \nabla (\theta_i T_l\vp)=\nabla (T_l\vp)$.

	 Let us notice that 
\[L_1^{l,n,\delta}=\frac{\lambda_1}{\lambda} \int_{Q_i}  M\left( x,   S_\delta\left( \frac{ E^n  - \sum_{i\in I}\nabla  (\theta_i T_l \vp)  }{ \lambda_1}\right) \right) \,dx dt.\]
Due to Lemma~\ref{lem:step2prev} the family of operators $S_\delta$ is uniformly bounded from $L_M(\Omega;\rn)$ to~$L_M(\Omega;\rn)$ and we can estimate $0\leq L^{l,n,\delta}_1\leq C L^{l,n,\delta}_3.$ Furthermore, Lemma~\ref{lem:dens} implies that $\lim_{n\to\infty} \lim_{\delta\to 0^+}L^{l,n,\delta}_3= 0,$  so  $\lim_{l\to\infty} \lim_{\delta\to 0^+} L^{l,n,\delta}_1= 0$ as well.

 Let us concentrate on $L^{l,n,\delta}_2$.
 The Jensen  inequality and then the Fubini theorem lead to 
	\begin{equation}\label{IE:aw17}
	\begin{split}
	&\frac{\lambda }{\lambda_2} L^{l,n,\delta}_2  = \int_{Q_i}  M\left( x,  \frac{ E^n (t,x)  -S_\delta E^n (t,x) }{\lambda_2} \right) \,dx dt\\
	& = \int_{Q_i} M \left( x, \frac{1}{ \lambda_2} \int_{B(0,\delta)} \varrho_\delta(y) \sum_{j=0}^n [  \mathds{1}_{E_j}(t,x) \va_j (t,x)- \mathds{1}_{E_j}(t,\kappa_\delta(x -  y)) \va_j (t,\kappa_\delta(x -  y)) ]\,dy \right) \,dx\,dt
	\\ &
	\leq 
	\int_{B(0,\delta)} \varrho_\delta(y)  \left( \int_{Q_i} M \left( x, \frac{1}{\lambda_2} \sum_{j=0}^n [   \mathds{1}_{E_n}(t,x) \va_j (t,x) -\mathds{1}_{E_j}(t,\kappa_\delta(x -  y)) \va_j(t,\kappa_\delta(x -  y)) ] \right) \,dx \right) \,dy\,dt.
	\end{split}	\end{equation}
Using the continuity of the shift operator in $L^1$ we observe that poinwisely
	\[  \sum_{j=0}^n [  \mathds{1}_{E_j}(t,x) \va_j (t,x)- \mathds{1}_{E_j}(t,\kappa_\delta(x -  y)) \va_j(t,\kappa_\delta(x -  y))   ] \xrightarrow[]{\dep\to 0} 0.\] 
	Moreover, when we fix arbitrary $\lambda_2>0$ we have 
	\[\begin{split}
	 &M \left( x, \frac{1}{\lambda_2} \sum_{j=0}^n [   \mathds{1}_{E_j}(t,x) \va_j(t,x) - \mathds{1}_{E_j}(t,\kappa_\delta(x -  y)) \va_j (t,\kappa_\delta(x -  y))  ] \right)\\
	 & \leq \sup_{\eta\in\rn:\ | \eta|=1}M \left( x, \frac{1}{ \lambda_2} \sum_{j=0}^n| \va_j|  \eta \right)  <\infty
	 \end{split}\]
and
the Lebesgue Dominated Convergence Theorem provides the right-hand side of \eqref{IE:aw17} converges to zero.

 To prove the convergence of $L^{l,n,\delta}_4$, which is independent of $\delta$ and $n$, we observe that when $l\to\infty$ we have $T_l \vp\to \vp$ strongly in $L^1(0,T; W_0^{1,1}(\Omega))$ and therefore also, up to a subsequence, almost everywhere. Moreover, $M(x,\nabla T_l \vp)\leq M(x,\nabla \vp)$ a.e. in~$\OT$. Consequently, the sequence $\{M(x,\nabla T_l \vp)\}_l$ is uniformly integrable. Taking into account its poinwise convergence, we infer modular convergence $\nabla T_l \vp \xrightarrow[l\to\infty]{M} \nabla \vp$. Thus, there exist a constant $\lambda_4$, such that $\lim_{l\to\infty} L_4^{l,n,\delta}=0$. 

Passing to the limit completes the proof of modular convergence of the approximating sequence from $L^\infty (0,T; C_c^\infty(\Omega))$. The modular convergence of gradients implies their strong $L^1$-convergence. \end{proof}

\section{Weak formulation}\label{ssec:Weak}

\begin{proof}[Proof of Lemma~\ref{lem:intbyparts}]  Let $h\in W^{1,\infty}(\r)$ be such that $\supp (h')$ is compact. Let us note that  $h_1,h_2:\R\to\R$ given by
\[h_1(t)=\int_{-\infty}^t (h')^+(s)\,ds,\ g^+:=\max\{0,g\},\quad h_2(t)=\int_{-\infty}^t (h')^-(s)\,ds,\ g^-:=\min\{0,g\}\]
are   Lipschitz continuous functions. Moreover, $h_1$ is non-decreasing, $h_2$ is non-increasing, and $h=h_1+h_2$. In both cases there exists $k>0$ such that $\supp(h')\subset [-k,k]$, thus $h(u)=h(T_k(u))=h_1(T_k(u))+h_2(T_k(u))$. Furthermore, $h_1(T_k(u)),h_2(T_k(u))\in L^\infty(\OT)$ and $\nabla (h_1(T_k(u))),\nabla (h_2(T_k(u)))\in L_M(\OT;\rn).$ It follows from the existence of modularly converging sequence $\nabla (T_k(u))_\ve$, cf.  Theorem~\ref{theo:approx}, which via Definition~\ref{def:convmod} implies uniform integrability of~\[\left\{M\left(x,\frac{h_1'((T_k (u))_\ve)\nabla (T_k(u))_\ve}{\lambda}\right)\right\}_\ve\] for some $\lambda>0$.

We start with the proof for nonnegative $\xi$, which we extend in the following way \begin{equation}\label{xi:ext}
\xi(t,x)=\left\{\begin{array}{ll}\xi(-t,x),& t<0,\\
\xi(t,x),&t\in[0,T],\\
0,&t>T.
\end{array}\right.\end{equation} Additionally, we extend $u(t,x)=u_0(x)$ for $t<0$.  Let further fix ${\et}>0$ and
\begin{equation}\begin{split}\label{t-zeta-eta}
 \zeta :=&h_1(T_k(u))\xi,\\
 \zeta_{\et}(t,x) := \frac{1}{{\et}}\int_t^{t+{\et}}\zeta(\sigma,x)\,d\sigma,\quad &\quad
 \wt{\zeta}_{\et}(t,x) := \frac{1}{{\et}}\int_{t-{\et}}^t\zeta(\sigma,x)\,d\sigma.\end{split}
\end{equation}
Note that due to the same reasoning as for $h(u)$, also $\zeta_{\et}, \wt{\zeta}_{\et}(t,x):\OT\to\r$ belong to $\VTMi$. Furthermore, $\partial_t \zeta_{\et},\partial_t \wt{\zeta}_{\et}(t,x)\in L^\infty(\OT)$. Then $\zeta_{\et}(T,x)= \wt{\zeta}_{\et}(0,x)=0$ for all $x\in\Omega$ and ${\et}>0$. We can use approximating sequences $((\zeta_{\et})_\varsigma)_\ve, ((\wt{\zeta}_{\et})_\varsigma)_\ve \in C_c^\infty (0,T; C_c^\infty(\Omega))$, where $\varsigma$ stands for mollification with respect to the time variable, and $\varepsilon$ denotes modular approximation from Theorem~\ref{theo:approx}  as  test functions  in~\eqref{eq:lem:int-by-parts-1}.  We get
\begin{equation}\label{1sttestve}
  \iOT A\cdot \nabla  (((\zeta_{\et})_\varsigma)_\ve)  \,dx\,dt-\iOT F  (((\zeta_{\et})_\varsigma)_\ve) \,dx\,dt =\iOT \left(u(t,x)-u_0(x)\right)\partial_t (((\zeta_{\et})_\varsigma)_\ve) \,dx\,dt.
\end{equation} 
  Since a modular convergence entails a weak one the Lebesgue Dominated Convergence Theorem enables to pass to the limit with $\varsigma,\ve\to 0$ on the right--hand side. On the left--hand side the properities of the regularising kernel together with Lemma~\ref{lem:NYS} ensures the convergence. In turn, we obtain
\begin{equation}\label{1sttest}
\begin{split}
\iOT A\cdot \nabla \zeta_{\et} \,dx\,dt-\iOT F \zeta_{\et} \,dx\,dt&=\iOT \left(u(t,x)-u_0(x)\right)\frac{1}{{\et}}\left(\zeta(t+{\et},x)-\zeta(t,x)\right)dxdt=\\
&= \frac{1}{{\et}}\left(J_1+J_2+J_3\right),
\end{split}
\end{equation}
where $\zeta(t,x)=0$ for $t>T$, $\xi$ is extended by~\eqref{xi:ext} $u(t,x)=u_0(x)$ for $t<0$, and
\begin{eqnarray}
J_1&=&\int_0^T\iO \zeta(t+{\et},x)u(t,x)dxdt=\int_{\et}^T\iO \zeta(t,x)u(t-{\et},x)dxdt\nonumber,\\
J_2&=&-\int_0^T\iO \zeta(t,x)u(t,x)dxdt,\label{J2}\\
J_3&=&-\int_0^T\iO \zeta(t+{\et},x)u_0(x)\,dxdt+\int_0^T\iO  \zeta(t,x) u_0(x)\,dxdt=\nonumber \\
&=&-\int_{\et}^{T+{\et}}\iO \zeta(t,x)u_0(x)\,dxdt+\int_0^T\iO  \zeta(t,x) u_0(x)\,dxdt=\nonumber \\
&=&\int_0^{\et}\iO \zeta(t,x)u_0(x)\,dxdt-\int_T^{T+{\et}}\iO  \zeta(t,x) u_0(x)\,dxdt\nonumber=\\
&=&\int_0^{\et} \iO \zeta(t,x)  u(t-{\et},x)dxdt.\label{J3}
\end{eqnarray}
Using~\eqref{J2} and~\eqref{J3} in~\eqref{1sttest} we get
\begin{equation}\label{1sttest-appl}
\begin{split}
 \iOT A\cdot \nabla \zeta_{\et} \,dx\,dt-\iOT F \zeta_{\et} \,dx\,dt&=\iOT \frac{1}{{\et}} \zeta(t,x) \left(u(t-{\et},x)-u(t,x)\right)dxdt.
\end{split}
\end{equation}
Note that for any $s_1,s_2\in\R$ we have
\begin{equation}
\label{h1-conv}
\int_{s_1}^{s_2} h_1(T_k(\s))d\s \geq  {h_1}(T_k(s_1)) (s_2-s_1).
\end{equation}
Then 
\begin{equation*}
\begin{split}
\frac{1}{{\et}} \iOT  \zeta(t,x) \left(u(t-{\et},x)-u(t,x)\right)dxdt\leq \frac{1}{{\et}} \iOT  \xi(t,x) \int_{u(t,x)}^{u(t-{\et},x)} h_1(T_k(\s))d\s\ dx\,dt.
\end{split}
\end{equation*}
Applying it in~\eqref{1sttest-appl}, following the same reasoning as in~\eqref{J3}, we get
\begin{equation}\label{est-integr2}
\begin{split}
& \iOT A\cdot \nabla \zeta_{\et} \,dx\,dt-\iOT F \zeta_{\et} \,dx\,dt\leq  \frac{1}{{\et}} \iOT  \xi(t,x) \left(\int_{u(t,x)}^{u(t-{\et},x)} h_1(T_k(\s))d\s\right) dx\,dt=\\
&= \frac{1}{{\et}} \iOT ( \xi(t+{\et},x)-\xi(t,x))\left(\int_{u(0,x)}^{u(t-{\et},x)} h_1(T_k(\s))d\s\right) dx\,dt .
\end{split}
\end{equation}
Passing to a subsequence if necessary, we have $\zeta_{\et}\xrightharpoonup[]{*}\xi h_1(T_k(u))$ weakly-* in $L^\infty(\OT)$, when ${\et}\searrow 0$. Since $\nabla \zeta_{\et}=[(\nabla \xi) h_1(T_k(u))]_{\et} +[\xi \nabla ( h_1(T_k(u)))]_{\et} $ and $[(\nabla\xi) h_1(T_k(u))]_{\et} \xrightharpoonup[]{*}(\nabla \xi) h_1(T_k(u))$ weakly-* in $L^\infty(\OT;\rn)$, when ${\et}\searrow 0$, by the Jensen inequality $\nabla \zeta_d \xrightarrow[]{M} \nabla (\xi h_1(T_k(u)))$. Moreover,$\zeta_d \xrightharpoonup[]{*} \xi h_1(T_k(u))$ weakly-* in $L^\infty(\OT)$, when ${\et}\searrow 0$. Therefore,  passing to the limit in~\eqref{est-integr2}  implies
\begin{equation}\label{est-integr-lim}
\begin{split}
& \iOT A\cdot \nabla (h_1(T_k(u))\xi) \,dx\,dt-\iOT F (h_1(T_k(u))\xi) \,dx\,dt\leq \iOT \partial_t\xi \int_{u_0}^{u(t,x)} {h_1}(T_k(\s))\,d\s\  dx\,dt.
\end{split}
\end{equation} 

Since $T_k(u_0)\in L^\infty(\Omega),$ there exists a sequence $\{u_{0}^{n}\}_n\subset{C_c^\infty}(\Omega)$ such that $T_k(u_{0}^{n})\to T_k(u_0)$ in~$L^1(\Omega)$ and a.e. in~$\Omega$ as $n\to\infty$. For $t<0$ and all $x\in\Omega$ we put $u(t,x)=u_{0}(x)$. Recall that we consider nonnegative $\xi\in  C_c^\infty([0,T)\times\Omega )$ extended by~\eqref{xi:ext}. Note that the sequence $ \{(\wt{\zeta}_{\et})_\ve\}$ approximating $\wt{\zeta}_{\et}$, given by~\eqref{t-zeta-eta}, can be used as a test function in~\eqref{eq:lem:int-by-parts-1}. Via arguments of~\eqref{1sttestve}, we pass to the limit with   $\ve\to 0$ getting
\begin{equation*}
\begin{split}
 \iOT A\cdot \nabla \wt{\zeta}_{\et} \,dx\,dt-\iOT F \wt{\zeta}_{\et} \,dx\,dt&=\iOT \frac{1}{{\et}}\left(\zeta(t,x)-\zeta(t-{\et},x)\right)\left(u(t,x)-u_0(x)\right)=\\
&= \frac{1}{{\et}}\left(K_1+K_2+K_3\right),
\end{split}
\end{equation*}
where  
\begin{eqnarray}
K_1&=&\int_0^T\iO \zeta(t ,x)u(t,x)dxdt=\int_{\et}^{T+{\et}}\iO \zeta(t-{\et},x)u(t-{\et},x)dxdt\nonumber,\\
K_2&=&-\int_0^T\iO \zeta(t-{\et},x)u(t,x)dxdt,\label{K2}\\
K_3&=&-\int_0^T\iO \zeta(t ,x)u_0(x)\,dxdt+\int_0^T\iO  \zeta(t-{\et},x) u_0(x)\,dxdt=\nonumber \\
&=& \int_{-{\et}}^0\iO \zeta(t,x)u_0(x)\,dxdt =\int_0^{\et} \iO \zeta(t-{\et},x)  u_0(x)dxdt.\label{K3}
\end{eqnarray}
Therefore,~\eqref{K2} and~\eqref{K3} give 
\begin{equation}\label{2ndtest}
\begin{split}
 \iOT A\cdot \nabla \wt{\zeta}_{\et} \,dx\,dt-\iOT F \wt{\zeta}_{\et} \,dx\,dt&= \frac{1}{{\et}}\left(L_1+L_2\right),
\end{split}
\end{equation}
with 
\begin{eqnarray*}
L_1&=&\int_{\et}^T\iO \zeta(t-{\et} ,x)(u(t-{\et},x)-u(t,x))\,dxdt\nonumber,\\
L_2&=&\int_0^{\et}\iO h_1(T_k(u_{0}^{n}))\xi(u(t-{\et},x)-u(t,x))\,dxdt+\nonumber\\
&&+\int_0^{\et}\iO (h_1(T_k(u_{0}))-h_1(T_k(u_{0}^{n})))\xi(u(t-{\et},x)-u(t,x))\,dxdt
\end{eqnarray*} 
for sufficiently small ${\et}$, because $\xi(\cdot,x)$ has a compact support in $[0,T)$ almost everywhere in~$\Omega$. Due to~\eqref{h1-conv}, we have 
\begin{equation}
\label{b-eta}
\begin{split}\int_{u(t-{\et},x)}^{u(t,x)}-(h_1(T_k(\s)))\,d\s&\leq -(u(t,x)-u(t-{\et},x))h_1(T_k(u(t-{\et},x)))\quad \text{ a.e. in }({\et},T)\times\Omega,\\
\int_{u(t-{\et},x)}^{u(t,x)}-(h_1(T_k(\s)))\,d\s  & \leq -(u(t,x)-u_0)h_1(T_k(u_0)) \quad  \text{ a.e. in }(0,{\et})\times\Omega.\end{split}
\end{equation}

Combining~\eqref{2ndtest} and~\eqref{b-eta} we get\begin{equation*}
\begin{split}
& \iOT A\cdot \nabla \wt{\zeta}_{\et} \,dx\,dt- \iOT F\wt{\zeta}_{\et} \,dx\,dt \geq    \frac{1}{{\et}} \iOT  \xi(t,x) \left(\int_{u(t,x)}^{u(t-{\et},x)} h_1(T_k(\s))d\s \right) dx\,dt+ \\
&\qquad +\int_0^{\et} \iO (h_1(T_k(u_{0}))-h_1(T_k(u_{0}^{n})))\xi(u_0-u(t,x)) dx\,dt\\ 
&\quad\geq  \frac{1}{{\et}} \iOT ( \xi(t-{\et},x)-\xi(t,x)) \left(\int_{u(t,x)}^{u(t-{\et},x)} h_1(T_k(\s))d\s \right) dx\,dt+ \\
&\qquad -  \iO |h_1(T_k(u_{0}^{n}))-h_1(T_k(u_{0}))||\xi|(|u_0|+|u(t,x)|) dx\,dt.
\end{split}
\end{equation*}

To pass with ${\et}\searrow 0$ and then $n\to\infty$ on the left-hand side above, as in~\eqref{est-integr-lim}, by the Lebesgue Dominated Convergence Theorem we obtain\begin{equation}\label{est-integr-lim-2}
\begin{split}
& \iOT A\cdot \nabla (h_1(T_k(u))\xi) \,dx\,dt-\iOT F (h_1(T_k(u))\xi) \,dx\,dt\geq\frac{1}{{\et}} \iOT \partial_t\xi  \int_{u_0}^{u(t,x)} {h_1}(T_k(\s))\,d\s\  dx\,dt.
\end{split}
\end{equation}

Combining~\eqref{est-integr-lim} with~\eqref{est-integr-lim-2} we conclude
\begin{equation}\label{end}
\begin{split}
& \iOT A\cdot \nabla (h_1(T_k(u))\xi) \,dx\,dt-\iOT F (h_1(T_k(u))\xi) \,dx\,dt=\frac{1}{{\et}} \iOT \partial_t\xi  \int_{u_0}^{u(t,x)} {h_1}(T_k(\s))\,d\s\  dx\,dt.
\end{split}
\end{equation}
for all nondecreasing and Lipschitz $h_1:\R\to\R$ and for all nonnegative $\xi$.

We can replace $h_1(T_k(u))$ by  $-h_2(T_k(u))$ in~\eqref{end} and in turn we can also replace it by  $h(T_k(u))=h(u).$ Since $\xi=\xi^++\xi^-$, where $\xi^+,\xi^-\in \VTMi$ we get the claim.\end{proof}

\section{Comparison principle}\label{ssec:Comp}

\begin{proof}[Proof of Proposition~\ref{prop:comp-princ}] Let us define two-parameter family of functions $\bt:\R\to\R$ by  
\begin{equation*}\bt(s):=\left\{\begin{array}{ll}
1& \text{for }s\in[0,\tau],\\
\frac{-s+\tau+r}{r}& \text{for }s\in[\tau,\tau+r],\\
0& \text{for }s\in[\tau+r,T]
\end{array}\right.
\end{equation*}
with arbitrary $\tau\in (0,T)$ and sufficiently small $r>0$, such that $\tau+r<T$, one-parameter family of~functions $H_\delta:\R\to\R$ by  
\[H_\delta(s)=\left\{\begin{array}{ll}
0,& s\leq 0,\\
s/\delta,& s\in(0,\delta),\\
1,& s\geq \delta,
\end{array}\right.\] 
with $\delta\in(0,1)$ and sets
\[Q_T^\delta=\{(t,x): 0<T_{l+1}(v^1)-T_{l+1}(v^2)<\delta\},\]
\[Q_T^{\delta+}=\{(t,x):  T_{l+1}(v^1)-T_{l+1}(v^2)\geq \delta\}.\]

Using (R2) with $h(v^1)=\psi_l(v^1)$, $\xi=H_\delta(T_{l+1}(v^1)-T_{l+1}(v^2))\bt(t)$ and $h(v^2)=\psi_l(v^2)$, $\xi=H_\delta(T_{l+1}(v^1)-T_{l+1}(v^2))\bt(t)$ and subtract the second from the first we get
 \[\begin{split}
&D^{\delta,r,l,\tau}_1+D^{\delta,r,l,\tau}_2+D^{\delta,r,l,\tau}_3+D^{\delta,r,l,\tau}_4+D^{\delta,r,l,\tau}_5=\\
=&-\int_{\OT} \left(\int_{v^1_0}^{v^2_0} \psi_l(\sigma)d\sigma+\int_{v^2}^{v^1} \psi_l(\sigma)d\sigma\right) \partial_t (H_\delta( T_{l+1}(v^1)-T_{l+1}(v^2))) \bt(t)\, dx\,dt+\\
&+\iO\frac{1}{r}\int_\tau^{\tau+r}  \left(\int_{v^1_0}^{v^2_0} \psi_l(\sigma)d\sigma+\int_{v^2}^{v^1} \psi_l(\sigma)d\sigma\right) H_\delta ( T_{l+1}(v^1)-T_{l+1}(v^2))\,dt\, dx +\\
&+\int_{Q_T^\delta}\frac{1}{\delta}(A(x,\nabla v^1)(\psi_l(v^1)-\psi_l(v^2))  \nabla ((T_{l+1}(v^1)-T_{l+1}(v^2))\bt(t)) \,dx\,dt+\\
&+\int_{Q_T^\delta}\frac{1}{\delta}(A(x,\nabla v^1)-A(x,\nabla v^2))\psi_l(v^2)  \nabla ((T_{l+1}(v^1)-T_{l+1}(v^2))\bt(t)) \,dx\,dt+\\
&+\int_{Q_T^\delta\cup Q_T^{\delta+}} (A(x,\nabla v^1)\nabla v^1\psi_l'(v^1)-A(x,\nabla v^2)\nabla v^2\psi_l'(v^2))  H_\delta(T_{l+1}(v^1)-T_{l+1}(v^2)) \bt(t) \,dx\,dt=\\
=&\int_{Q_T^\delta\cup Q_T^{\delta+}} (f^1 \psi_l(v^1)-f^2 \psi_l(v^2)) H_\delta(T_{l+1}(v^1)-T_{l+1}(v^2)) \bt(t)\,dx\,dt=D^{\delta,r,l,\tau}_R.
\end{split}\] We observe that  
\[\begin{split}|D^{\delta,r,l,\tau}_1|\leq& \int_{Q_T^\delta} \left|\partial_t\left(\int_{v^1_0}^{v^2_0} \psi_l(\sigma)d\sigma+\int_{v^2}^{v^1} \psi_l(\sigma)d\sigma\right)   \frac{1}{\delta}( T_{l+1}(v^1)-T_{l+1}(v^2))  \bt(t)\right| dx\,dt+\\
&+\int_{Q_T^\delta}\left|\left(\int_{v^1_0}^{v^2_0} \psi_l(\sigma)d\sigma+\int_{v^2}^{v^1} \psi_l(\sigma)d\sigma\right)  \frac{1}{\delta}( T_{l+1}(v^1)-T_{l+1}(v^2)) \partial_t\bt(t)\right| dx\,dt=\\
=&\int_{Q_T^\delta} \left|\left(\psi_l(\partial_t (T_{l+1}(v^1))-\psi_l(\partial_t T_{l+1}(v^2)) \right)  
 \frac{1}{\delta}\cdot\delta  \right| dx\,dt+\\
&+\int_{Q_T^\delta}\left|\left(T_{l+1}(v^1_0)-T_{l+1}(v^2_0)+T_{l+1}(v^1)-T_{l+1}(v^2)\right)  \frac{1}{\delta}\cdot\delta\right|  dx\,dt\leq (2+4(l+1))|Q_T^\delta|.\end{split}\]
Hence,  the Dominated Convergence Theorem yields that $D^{\delta,r,l,\tau}_1\to 0$ when $\delta\to 0$. 
In the case of $D^{\delta,r,l,\tau}_2$, $D^{\delta,r,l,\tau}_5$, and $D^{\delta,r,l,\tau}_R$ it also suffices to apply the Dominated Convergence Theorem. Furthermore, the monotonicity of~truncations implies $D^{\delta,r,l,\tau}_3\geq 0$, whereas the monotonicity of $A$ implies $D^{\delta,r,l,\tau}_4\geq 0$. 
We erase nonnegative terms on the left-hand side and pass to the limit with $\delta\to 0$ in the remaining ones, getting
 \[\begin{split}
&\lim_{\delta\to 0}D^{\delta,r,l,\tau}_2+\lim_{\delta\to 0}D^{\delta,r,l,\tau}_5=D^{ r,l,\tau}_2+ D^{ r,l,\tau}_5=\\
=&\iO\frac{1}{r}\int_\tau^{\tau+r}  \left(\int_{v^1_0}^{v^2_0} \psi_l(\sigma)d\sigma+\int_{v^2}^{v^1} \psi_l(\sigma)d\sigma\right) \sg ( T_{l+1}(v^1)-T_{l+1}(v^2))\,dt\, dx +\\
&+\iOT (A(x,\nabla v^1)\nabla v^1\psi_l'(v^1)-A(x,\nabla v^2)\nabla v^2\psi_l'(v^2))  \sg(T_{l+1}(v^1)-T_{l+1}(v^2)) \bt(t) \,dx\,dt\leq\\
\leq&\iOT (f^1 \psi_l(v^1)-f^2 \psi_l(v^2)) \sg(T_{l+1}(v^1)-T_{l+1}(v^2)) \,dx\,dt=\lim_{\delta\to 0}D^{\delta,r,l,\tau}_R.
\end{split}\]

What is more, due to~\eqref{eq:contr:rad:n} and uniform boundedness of the rest expression in $D^{r,l,\tau}_5$, we infer that $\lim_{l\to\infty}D^{r,l,\tau}_5=0.$ The Monotone Convergence Theorem enables to pass with $l\to\infty$ also in $D^{r,l,\tau}_2$ and $D^{r,l,\tau}_5$. Consequently, we obtain
 \[ \iO\frac{1}{r}\int_\tau^{\tau+r}  \left( {v^2_0}-{v^1_0} +{v^1}-{v^2} \right) \sg ( v^1-v^2)\,dt\, dx \leq \iOT (f^1  -f^2  ) \sg( v^1 -v^2) \,dx\,dt.\]
 Since a.e. $\tau\in[0,T)$ is the Lebesgue point of the integrand on the left-hand side and we can pass with $r\to 0$. After rearranging terms it results in
\[\begin{split}&\iO(v^1(\tau,x)-v^2(\tau,x))\sg(v^1(\tau,x)-v^2(\tau,x))dx\leq\\
&\leq\iOT(f^1(t,x)-f^2(t,x))\sg(v^1(\tau,x)-v^2(\tau,x))dx\,dt\\
&+\iO(v_0^1(\tau,x)-v_0^2(\tau,x))\sg(v^1(\tau,x)-v^2(\tau,x))dx.\end{split},\]
for a.e. $\tau\in (0,T)$. Note that the left-hand side is nonnegative. Since $f^2\geq f^1$ and $v_0^2\geq v_0^1$, the right-hand side is nonpositive. Hence,  $\sg(v^1(\tau,x)-v^2(\tau,x))=0$ for a.e. $\tau\in (0,T)$ and~consequently $v^1 \leq v^2 $ a.e. in~$\OT$.\end{proof}
\bibliographystyle{plain}
\bibliography{PGISAZG-para-arxiv}
\end{document}